\documentclass[10pt]{article}
 \usepackage{amsfonts, epsfig, amsmath, amssymb, color,amsthm,mathabx}
\usepackage{mathrsfs}
\usepackage[english]{babel}

\usepackage[round]{natbib}   

\newcommand{\E}{\mathbf{E}}
\renewcommand{\P}{\mathbf{P}}

\usepackage{color}

\renewcommand {\epsilon}{\varepsilon}

\theoremstyle{plain}
\newtheorem{thm}{Theorem}[section]
\newtheorem{lem}[thm]{Lemma}

\newtheorem{exa}[thm]{Example}

\theoremstyle{definition}
\newtheorem{rem}[thm]{Remark}

\usepackage{wrapfig}
\usepackage{tikz}
\usepackage{pgfplots}
\usetikzlibrary{arrows.meta,calc}

\DeclareMathSymbol{\ophi}{\mathalpha}{letters}{"1E}

\newcommand{\e}{\varepsilon}

\renewcommand{\phi}{\varphi}

\newcommand{\be}{\begin{equation}}
\newcommand{\ee}{\end{equation}}
\newcommand{\ben}{\begin{equation*}}
\newcommand{\een}{\end{equation*}}

\newcommand{\ba}{\begin{equation}\begin{aligned}}
\newcommand{\ea}{\end{aligned}\end{equation}}

\DeclareMathOperator{\Var}{Var}

\newcommand{\di}{\mathrm{d}}

\newcommand{\rF}{\mathscr{F}}

\newcommand{\bI}{\mathbb{I}}

\newcommand{\bN}{\mathbb{N}}
\newcommand{\bR}{\mathbb{R}}

\allowdisplaybreaks[4]

\let\oldmarginpar\marginpar
\renewcommand{\marginpar}[1]{\oldmarginpar{\scriptsize\texttt{\color{red}{#1}}}}

\numberwithin{equation}{section}

\begin{document}
\title{Diffusion limits of cyclic finite-velocity random motions\\ along vector fields}

\date{\today}

\author{Olga Aryasova\footnote{Institute of Mathematics, Friedrich Schiller University Jena, Inselplatz 5,
07743 Jena, Germany,
and Institute of Geophysics, National Academy of Sciences of Ukraine, Palladin Ave.\ 32, Kyiv
03680, Ukraine, and
Igor Sikorsky Kyiv Polytechnic Institute, Beresteiskyi Ave.\ 37, Kyiv 03056, Ukraine;
oaryasova@gmail.com},
Ilya Pavlyukevich\footnote{Institute of Mathematics, Friedrich Schiller University Jena,
Inselplatz 5,
07743 Jena, Germany; ilya.pavlyukevich@uni-jena.de}, and
Yusif Sualah\footnote{Institute of Mathematics, Friedrich Schiller University Jena, Inselplatz 5,
07743 Jena, Germany; yusif.sualah@uni-jena.de}}

\maketitle

\begin{abstract}
We investigate diffusion approximations for a class of multivariate inhomogeneous finite-velocity
random motions motivated by models of active particles. A particle alternates cyclically among prescribed velocity fields
$V_1,\dots,V_p$ with random run times and evolves according to either a \emph{flight} dynamics,
consisting of piecewise-linear motion, or a \emph{gliding} dynamics, in which the particle follows the corresponding velocity flow.
Under a Kac-type scaling that couples vanishing run times with diverging particle speed,
we prove weak convergence of both processes to multidimensional diffusions.
It turns out that the order of successive cyclic motions affects the diffusion limit.
In addition to the second-order term generated by fluctuations of the random run times,
the limiting drift contains directional derivatives $DV_i[V_j]$
of the underlying vector fields. In Stratonovich form, part of this drift is expressed through their Lie brackets $[V_i,V_j]$.
Thus, the generic non-commutativity of the microscopic motions survives the diffusive scaling and generates a macroscopic drift
which depends on the cyclic order.
The limiting drift also distinguishes the flight and gliding dynamics, despite their
being driven by the same vector fields and run times.
Several examples, including run-and-reverse motion and cyclic dynamics generated by multiple vector fields,
illustrate how the cyclic switching protocol and the geometry of
the underlying velocity fields shape the limiting diffusions.
\end{abstract}


\smallskip


\begin{footnotesize}\tableofcontents\end{footnotesize}

\section{Introduction}

The present work is motivated by the physical paradigm of active particles,
i.e., particles capable of self-propelled, directed motion. Their dynamics is often
characterized by alternating phases of persistent motion and reorientation.
Typical examples include motile bacteria, sperm cells, and artificial microswimmers.
At the level of a single particle, such dynamics naturally leads to finite-velocity stochastic
models in which the particle moves persistently for a random period of time and subsequently
changes its direction of motion.

The classical and simplest example of such a finite-velocity random motion
is provided by the telegraph process $X$ which describes the one-dimensional motion
of a particle that moves with constant speed $v\in(0,\infty)$ and reverses its direction of
motion at the arrival times of a Poisson process. It is defined by the equation
\ba
\label{e:t1}
X_t = x + v\xi \int_0^t (-1)^{N_s}\,\di s,
\ea
where $N=(N_t)_{t\in[0,\infty)}$ is a Poisson process of intensity $c\in(0,\infty)$,
while $x\in\mathbb{R}$ and $\xi\in\{-1,1\}$ denote the initial position and the initial direction of motion,
respectively. Although the position process $(X_t)_{t\in[0,\infty)}$ is not Markovian,
the joint process $(X_t,(-1)^{N_t})_{t\in[0,\infty)}$ consisting of the position and the velocity is Markov.

A remarkable property of the telegraph process
is that the probability density $p(t,x)$ of the position $X_t$ satisfies the hyperbolic, rather than parabolic,
partial differential equation
\begin{equation}
\label{e:telegraph}
\frac12\frac{\partial^2}{\partial t^2}p
+c\frac{\partial}{\partial t}p
=\frac{v^2}{2}\frac{\partial^2}{\partial x^2}p,
\end{equation}
known as the \emph{telegraph equation}. \cite{goldstein1951diffusion} derived this equation in the
context of diffusion emerging from continuous finite-velocity motion,
addressing a question originally posed by \cite{taylor1922diffusion}. Later, \cite{kac1956some,kac1974stochastic}
introduced the telegraph process $X$ with the help of \eqref{e:t1} to obtain a probabilistic representation of
the solution to \eqref{e:telegraph}.

Goldstein and Kac also observed that, under the scaling
$c,v\to\infty$ with $v^2/c\to\varkappa^2\in(0,\infty)$,
now commonly referred to as \emph{Kac's scaling}, the telegraph equation \eqref{e:telegraph} converges to the heat equation
\begin{equation}
\frac{\partial}{\partial t}p
=\frac{\varkappa^2}{2}\frac{\partial^2}{\partial x^2}p,
\end{equation}
whose fundamental solution is the transition probability density of Brownian motion with variance parameter $\varkappa^2$.
This observation suggests that, under Kac's scaling, the telegraph process itself converges to Brownian motion in an appropriate sense.
Thus, the telegraph process provides a basic example of a finite-velocity microscopic dynamics giving rise,
under rapid direction switching, to an effective diffusive motion.

The clear physical motivation, i.e., the study of a random motion with finite propagation speed, and the mathematical
elegance of the telegraph process have continued to attract considerable attention ever since its introduction.
Over the years, the process $X$ and its modifications have appeared in the literature under a variety of names,
including \emph{telegraph process}, \emph{telegraph random evolution}, \emph{Goldstein--Kac process},
\emph{transport process}, \emph{finite-velocity random motion} (or \emph{flights}),
\emph{velocity-jump process}, \emph{run-and-tumble} (or \emph{run-and-turn}) process,
\emph{persistent random walk}.
Related finite-velocity models are also studied within the framework of \emph{continuous-time random walks}
and, in particular, L\'evy walks, see the survey by \cite{zaburdaev2015levy}.

The development of the telegraph process also stimulated the emergence of the general
theory of \emph{random evolutions}, a term apparently introduced by \cite{griego1969random,griego1971theory}.
Classical references include \cite{hersh1974random}, \cite[Chapter~12]{EthierK-86},
\cite{pinsky1991lectures}, and the monographs by \cite{swishchuk1997random,swishchuk2000random}.
For a comprehensive account of one-dimensional telegraph processes, their probabilistic properties,
and applications, particularly in mathematical finance, we refer the reader to the monograph by
\cite{kolesnik2013telegraph}.

The numerous generalizations of the telegraph process can be roughly divided into several broad classes.

One important direction concerns multidimensional finite-velocity motions in which the direction of travel is random,
typically uniformly distributed over the unit sphere. Already \cite{monin1956statistical} independently
introduced such models to describe particle scattering in homogeneous isotropic media.
Transport processes with finitely many admissible velocities evolving according to
a finite-state Markov chain were investigated in \cite{pinsky1968differential,kurtz1973limit},
while isotropic transport processes on manifolds were studied by \cite{pinsky1976isotropic}.
More recently, random motions in $\mathbb{R}^3$ with orthogonal directions were investigated by \cite{cinque2023random},
multidimensional cyclic motions with $n$ directions in \cite{lachal2006cyclic,lachal2006minimal},
cyclic four-directional motions by \cite{leorato2003alternating,orsingher2004cyclic,iuliano2024cyclic},
and more general multidimensional models by \cite{cinque2024multidimensional}.

Another central research theme concerns diffusion approximations. A rigorous proof of the weak convergence of the telegraph
process \eqref{e:t1} to Brownian motion under Kac's scaling is given in Chapter~12.1 of \cite{EthierK-86}.
Quantitative convergence estimates in Wasserstein distance were recently obtained by \cite{barrera2023quantitative},
while weak convergence in the more general semi-Markov setting was established by \cite{pedicone2026weak}.
Weak diffusion limits of various types of multivariate transport processes
were subsequently established by
\cite{watanabe1968weak,watanabe1970convergence,pinsky1968differential,kurtz1973limit,
tutubalin1967central,gorostiza1972central,gorostiza1973invariance,gorostiza1975convergence,othmer1988models,
othmer2000diffusion,othmer2002diffusion,hillen2016diffusion}.
Strong approximation results include the almost sure convergence obtained by \cite{Griego71}
via the Skorokhod embedding theorem and multivariate strong approximations developed by \cite{bashtova2025strong}.

Finally, a further line of research concerns inhomogeneous telegraph evolutions. Time-in\-ho\-mo\-ge\-neous switching
rates were first considered by \cite{kaplan1964differential}, who derived the corresponding
time inhomogeneous telegraph equation. Random evolutions with state-dependent velocities
were discussed in the survey of \cite{hersh1974random}. Telegraph processes with time and space-dependent velocities
and their diffusion limits were analysed in dimensions one and two by \cite{ratanov1999telegraph,orsingher2008random,orsingher2026one}.
\cite{benaim2015qualitative,benaim2019random} studied inhomogeneous random evolutions in the framework of
piecewise deterministic Markov processes.

The present work lies at the intersection of the latter two areas and is devoted to diffusion
approximations for multivariate cyclic finite-velocity random motions along vector fields.
The cyclic and inhomogeneous motions considered here may be viewed as models of active particles
moving in spatially inhomogeneous environments determined by vector fields $V_1,\ldots,V_p$ and changing
their directions according to a prescribed cyclic reorientation rule. We distinguish between two types
of motion: \emph{gliding} along the vector fields and straight \emph{flights} with the
initial velocity of each run determined by the corresponding vector field. The random run times may
have general $(2+\rho)$-integrable distributions and, in particular, need not be exponentially
distributed, placing the model in a non-Markovian renewal setting.

The order of the vector fields together with the distributions of the run times determines the
microscopic switching dynamics, whereas the diffusion approximation under Kac's scaling and a
centering condition provides an effective description on the macroscopic scale. Besides obtaining
explicit formulas for the limiting diffusions, a principal feature of our results is that the \emph{ordering}
of the cyclic dynamics survives in the diffusion limit through an additional drift determined by directional
derivatives $DV_i[V_j]$ of the underlying vector fields. In the Stratonovich representation, these terms give rise,
in particular, to Lie-bracket contributions $[V_i,V_j]$. Thus, the non-commutativity and orientation of
the microscopic cyclic dynamics may remain visible at the macroscopic scale.

This effect is illustrated in Example~\ref{exa:UV}. We consider motions generated by two vector
fields $U$ and $V$ and their opposites. The ``counterclockwise'' cyclic ordering
$(U,V,-U,-V)$
and the ``clockwise'' ordering
$(U,-V,-U,V)$
produce the same quadratic fluctuations but opposite Lie-bracket components in the limiting drift.
Moreover, the flight and gliding dynamics generally give rise to different limiting drifts,
despite being driven by the same vector fields and run times.

The diffusion approximation obtained here can also be viewed from the perspective of stochastic homogenization.
The short successive run times introduce a rapidly varying microscopic time scale,
while Kac's scaling simultaneously increases the magnitude of the velocities.
The centering condition eliminates the first-order averaged motion, and the resulting diffusion
can therefore be interpreted as a second-order effective dynamics of the underlying cyclic
random evolution. In contrast to a purely averaging regime, the effective dynamics retains
information about the ordered composition of the spatially dependent vector fields through the drift terms described above.

The appearance of order-dependent Lie-bracket drift also bears a connection to second-order (area)
corrections arising in Wong--Zakai approximations and rough-path invariance principles.
In particular, microscopic ordering produces second-order geometric information that survives the passage to the limit
and acts through Lie brackets of the driving vector fields. Related effects arise in
Wong--Zakai approximation limits, see \cite[Theorem 7.2 in Chapter 7]{IW89},
and in the area-anomaly phenomenon; see, e.g., \cite{lopusanschi2020area,diamantakis2025levy}.
We mention this connection only as an interpretation of the order-dependent drift; the convergence
results of the present paper are established by classical weak-convergence methods.

The paper is organized as follows.
In section \ref{s:setting}, we introduce the model, state the assumptions and main results, and
discuss several examples illustrating the influence of cyclic ordering and the geometry
of the vector fields on the limiting diffusion. Sections \ref{s:Awrc}
through \ref{s:Aconv} are devoted to the proof of the diffusion approximation for the flight dynamics,
while the corresponding result for the gliding dynamics is established in section \ref{s:gliding}.

Throughout this paper, we use the following notation.
The norm and scalar product in Euclidean spaces are denoted by $|\cdot|$ and
$\langle\cdot,\cdot\rangle$, respectively.
For a function $f=f(x)$, $\|f\|$ denotes
its supremum norm. The gradient and the Hessian matrix of a real-valued function $f$ are denoted by $\nabla f$ and $\nabla^2f$,
respectively.
For two nonnegative real-valued functions $f$ and $g$, we write $f(x)\leq_C g(x)$ if there exists a constant $C\in (0,\infty)$ such that
$f(x)\leq C g(x)$ for all $x$, and this constant will not be referenced further.
When $f$ and $g$ depend on additional parameters (e.g., $t$, $h$, etc.)
the inequality is assumed to hold uniformly over those parameters.

\section{Inhomogeneous multivariate finite-velocity motions: setting and main results\label{s:setting}}

We consider an idealized model of an active particle (microswimmer) performing cyclic motions in a spatially inhomogeneous environment.
Let the state space be $d$-dimensional, $d\in\mathbb N$,
let $p\in\mathbb{N}$ denote the \emph{period} of the cyclic motion, and
let $V_1,\dots, V_{p}\colon \bR^d\to \bR^d$ be the vector fields describing the possible directions of motion.
The travel times (run times) related to the vector fields $V_1,\dots, V_{p}$ are
determined by independent copies of positive random variables $\tau_1,\dots,\tau_p$.

We distinguish between two types of motion: \emph{flights}, corresponding to straight motion
with velocity frozen at the beginning of each run, and \emph{gliding},
corresponding to motion along the flow generated by the current vector field, see Fig.~\ref{f:paths}.

1. Flights.
Assume that the particle starts at $t=0$ at position $X_0=x$ with the initial ``direction'' $n_0=\iota\in\{1,\dots,p\}$.
Then, it moves with the constant velocity $ V_\iota(x)$ during the random time interval $T_1\stackrel{\di}{=}\tau_\iota$.
This motion is straight, and the particle's velocity remains constant throughout the first run and is determined by
$V_\iota(x)$.
At the time instant $T_1$, the particle changes its velocity
and continues moving according to the law
\ba
X_t&=X_{T_1}+V_{n_1}(X_{T_1})(t-T_1),\quad  t\in[T_1,T_2),\\
\ea
where
$n_1=1+ \iota \text{ mod}\,p$
and
$T_2-T_1\stackrel{\di}{=}\tau_{n_1}$ is independent of $T_1$.

Overall, the flight dynamics has the form
\ba
&X_0=x\in\bR^d,\quad n_0=\iota \in \{1,\dots, p\},\quad T_0=0,\\
&n_k= 1+(\iota + k-1)  \,\mathrm{ mod}\, p,\\
&T_{k+1}-T_k\stackrel{\di }{=}\tau_{n_{k}} \text{ are independent random variables},\\
&X_{t} = X_{T_k} + V_{n_k}(X_{T_k}) (t-T_k),\quad t\in[T_k,T_{k+1}), \quad  k\in\mathbb N_0.
\ea
Here, the process $n=(n_k)_{k\in\mathbb N_0}$ parameterizes the cyclic motion
and $V_{n_k}(\cdot)$ is the velocity field on step $k$.

2. Gliding. In this case, the particle's motion is not straight
but follows the velocity current. More precisely, the gliding dynamics has the form
\ba
&Y_0=y\in\bR^d,\quad n_0=\iota \in \{1,\dots, p\},\quad T_0=0,\\
&n_k= 1+(\iota + k-1)  \,\mathrm{ mod}\, p,\\
&T_{k+1}-T_k\stackrel{\di }{=}\tau_{n_{k}} \text{ are independent random variables},\\
&Y_{t} = Y_{T_k} + \int_{T_k}^tV_{n_{k}} (Y_s)  \,\di s ,\quad t\in[T_k,T_{k+1}), \quad  k\in\mathbb N_0.
\ea

\begin{figure}
\begin{center}
 \includegraphics{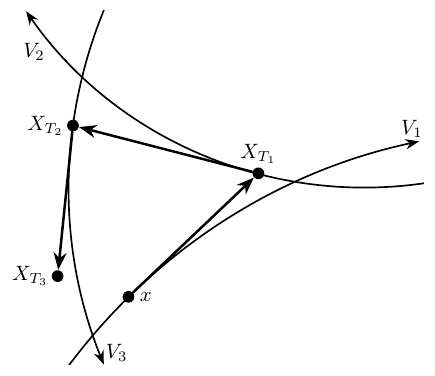}\hspace{.5cm}
\includegraphics{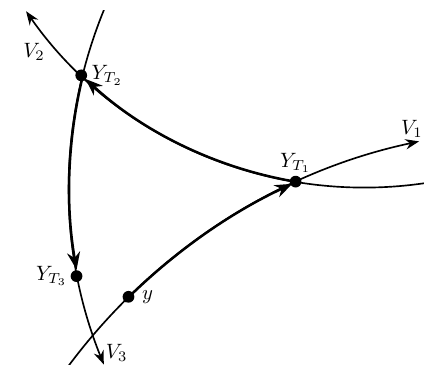}
\end{center}
\caption{Schematic trajectories of the flight (left) and gliding (right) dynamics generated by three vector fields.\label{f:paths}}
\end{figure}

\begin{exa}[one-dimensional inhomogeneous telegraph evolutions]
\label{exa:1}
Let $V\colon \bR\to\bR$, $V_1=V$, $V_2=-V$, $N=(N_t)_{t\in [0,\infty)}$
be a Poisson process of intensity $c\in (0,\infty)$, and
let $(T_k)_{k\in\mathbb N_0}$ be its arrival times.
Then, $X_0=x$,
\ba
X_{t} = X_{T_k} + (-1)^{\iota + k-1} V(X_{T_k}) (t-T_k),\quad t\in[T_k,T_{k+1}), \quad  k\in\mathbb N_0,
\ea
is the inhomogeneous one-dimensional ``flight'' telegraph evolution, $\iota\in\{1,2\}$ parameterizing the initial direction of motion.

Analogously,
\ba
Y_{t} = y + (-1)^{\iota-1} \int_0^t V(Y_s)(-1)^{N_s}\,\di s
\ea
is inhomogeneous one-dimensional ``gliding'' telegraph evolution, as considered in \cite{ratanov1999telegraph}.

Note that if $V(x)\equiv v\in (0,\infty)$, the evolutions $X$ and $Y$ coincide with the process \eqref{e:t1} with
$\xi=(-1)^{\iota-1}$.
\end{exa}

In general, the dynamics of $X$ and $Y$ may be very complicated.
Our goal is to study the diffusion limits of $X$ and $Y$ in a regime
where the particle speed increases while the run times $T_{k+1}-T_k$, $k\in\mathbb N_0$, become short.

We make the following assumptions on the building blocks of the systems.

\noindent
$\mathbf{A}_V$: Let $p\in\mathbb N$. The vector fields $V_1,\dots,V_{p}$ are bounded with bounded first and second derivatives, i.e.,
\ba
V_i\in C_b^2(\bR^d,\bR^{d}),\quad i=1,\dots,p.
\ea
We denote $\|V\|=\max\{\|V_1\|,\dots,\|V_p\|\}$ and $\|\nabla V\|=\max\{\|\nabla V_1\|,\dots,\|\nabla V_p\|\}$.

\noindent
$\mathbf{A}_\tau$:
The random variables $\tau_1,\dots,\tau_{p}$ are positive and the Lyapunov condition
$\E \tau_i^{2+\rho}<\infty$, $i=1,\dots, p$, is satisfied for some $\rho\in (0,1)$.
We denote
\ba
\E \tau_i   &=\mu_i\in (0,\infty),\\
\E \tau_i^2 &=\lambda_i^2\in (0,\infty),\\
\Var \tau_i &=\lambda_i^2-\mu_i^2=\sigma_i^2\in[0,\infty),\quad  i=1,\dots, p.
\ea
We also set
\ba
\mu:=\mu_1+\cdots+\mu_{p}.
\ea

In order to prevent blow-ups when passing to the diffusion limit, we make the following
balance assumption on the vector fields $V_1,\dots,V_{p}$ and mean travel times  $\mu_1,\dots,\mu_{p}$.

\noindent
$\mathbf{A}_{V,\mu}$:
The vector fields $V_1,\dots,V_{p}$ satisfy
\ba
\label{e:Vmu}
\mu_1 V_1(x)+\cdots+\mu_{p}V_{p}(x)=0,\quad x\in\mathbb R^d.
\ea

For the constant vector fields $V_i(x)\equiv V_i$, $i=1,\dots,p$,
condition \eqref{e:Vmu} means that the expected displacement over one complete cycle vanishes:
\ba
\E\Big[V_1\tau_1 +\cdots+V_{p}\tau_{p} \Big]=\sum_{i=1}^p \mu_i V_i=0.
\ea
Thus, \eqref{e:Vmu} is a natural centering condition eliminating the first-order averaged motion.

\begin{exa}
At first sight, assumption $\mathbf{A}_{V,\mu}$ may appear rather restrictive and applicable
only to a narrow class of vector fields $V_1,\dots,V_p$. The following example shows that it naturally
encompasses a broad family of models.

Let $V\in C_b^2(\mathbb R^d,\mathbb R^{d\times l})$ be a matrix-valued function, and let
$\xi_1,\dots,\xi_p\in\mathbb R^l$ be fixed vectors, which we refer to as \emph{controls}. At each step,
the particle evolves according to the vector field
\ba
V_i(x):=V(x)\xi_i,\quad i=1,\dots,p.
\ea
In this setting, assumption $\mathbf{A}_{V,\mu}$ is equivalent to the following condition on the control vectors:
\ba
\mu_1\xi_1+\cdots+\mu_p\xi_p=0.
\ea
Thus, the centering condition can be imposed at the level of the finite collection of controls
$\xi_1,\dots,\xi_p$ independently of the spatial dependence encoded by $V(x)$.
\end{exa}

Let $\e,\delta\in(0,1]$ be the small parameters responsible for the space-time scaling of the evolutions $X$ and $Y$. Then,
we consider the scaled flight dynamics:
\ba
\label{e:Xe}
&X_0^{\e,\delta}=x\in\bR^d,\quad \iota \in \{1,\dots, p\},\quad T^\e_0=0,\\
&n_k=1+ (\iota + k-1)  \,\mathrm{ mod}\, p,\\
&T^\e_{k+1}-T^\e_k\stackrel{\di }{=}\e \tau_{n_{k}},\\
&X_{t}^{\e,\delta} = X^{\e,\delta}_{T_k^\e}
+ \frac{1}{\delta}V_{n_{k}}(X^{\e,\delta}_{T_k^\e}) (t-T^\e_k),\quad t\in[T^\e_k,T^\e_{k+1}),
\quad  k\in\mathbb N_0,
\ea
and the scaled gliding dynamics
\ba
&Y_0^{\e,\delta}=y\in\bR^d,\quad \iota \in \{1,\dots, p\},\quad T^\e_0=0,\\
&n_k=1+ (\iota + k-1)  \,\mathrm{ mod}\, p,\\
&T^\e_{k+1}-T^\e_k\stackrel{\di }{=}\e \tau_{n_{k}},\\
&Y_{t}^{\e,\delta} = Y^{\e,\delta}_{T_k^\e} + \frac{1}{\delta}\int_{T_k^\e}^t V_{n_{k}} (Y^{\e,\delta}_s) \,\di s ,
\quad t\in[T_k^\e,T_{k+1}^\e), \quad  k\in\mathbb N_0.
\ea
In both models, the velocity is of order $\mathcal{O}(1/\delta)$, while the duration of each run is of order $\mathcal{O}(\e)$.
Hence, the displacement during a single run is of order $\mathcal{O}(\e/\delta)$. Under the centering condition \eqref{e:Vmu},
the first-order contributions cancel over a complete cycle, and a non-trivial fluctuation limit is obtained under
the Kac-type scaling. This leads to the following assumption.

\noindent
$\mathbf{A}_\varkappa$: There is $\varkappa\in (0,\infty)$ such that there is a limit
\ba
\lim_{\e,\delta\downarrow 0 }\frac{\e}{\delta^2}=\varkappa^2.
\ea
Note that under assumption $\mathbf{A}_\varkappa$ we have
\ba
\frac{\e}{\delta}= \delta \frac{\e}{\delta^2}=\mathcal O(\delta)=\mathcal O(\sqrt{\e}),\quad \e,\delta\downarrow 0 .
\ea

For the formulation of the main result we will need the directional derivatives $D V_i[V_j]=DV_i(x)V_j(x)$ of the vector fields $V_i$
in direction $V_j$ defined as
vector fields with coordinates
\ba
\label{e:DVV}
(D V_i[V_j])^\alpha = \sum_{\beta=1}^d V_j^\beta   \cdot  \partial_\beta V_i^\alpha,\quad \alpha=1,\dots, d,\  1\leq i,j\leq p,
\ea
and the Lie bracket $[V_i,V_j]$ defined as
\ba
{} [V_i,V_j]&= D V_j [ V_i] - D V_i [ V_j],\quad 1\leq i,j\leq p,\\
([V_i,V_j])^\alpha &
= \sum_{\beta=1}^d  \Big(V_i^\beta  \cdot  \partial_\beta V_j^\alpha
- V_j^\beta \cdot \partial_\beta V_i^\alpha\Big) ,\quad \alpha=1,\dots, d,\  1\leq i,j\leq p.
\ea

The main results are presented in the following theorems.

\begin{thm}
\label{t:A}
Suppose that assumptions
$\mathbf{A}_V$,
$\mathbf{A}_\tau$,
$\mathbf{A}_{V,\mu}$ and
$\mathbf{A}_\varkappa$ hold. Then, for any initial values $x\in\bR^d$ and $\iota\in \{1,\dots,p\}$,
the flight evolutions $\{X^{\e,\delta}\}$ converge weakly in $C(\bR_+,\bR^d)$
to a diffusion $X$ starting at $x$ with the generator
\ba
\label{e:opA}
Af
= \frac{\varkappa^2}{2\mu}  \sum_{i} \sigma_{i}^2 \langle \nabla^2 f  V_i, V_i\rangle
+ \frac{\varkappa^2}{\mu} \sum_{j< i} \mu_i\mu_j \langle \nabla f , D V_i [ V_j] \rangle,\quad f\in C^3_b(\bR^d,\bR).
\ea
The diffusion $X$ is a solution to an It\^o-SDE
\ba
\label{e:SDEX}
\di X_t= \frac{\varkappa}{\sqrt\mu}  \sum_{i} \sigma_{i} V_i(X_t)\,\di W^i_t +
\frac{\varkappa^2}{\mu} \sum_{j< i} \mu_i\mu_j  D V_i [ V_j](X_t)\,\di t,
\ea
where $W^1,\dots,W^p$ are independent standard one-dimensional Brownian motions.
The SDE \eqref{e:SDEX} can be rewritten in the Stratonovich form
as
\ba
\label{e:XStrat}
\di X_t= \frac{\varkappa}{\sqrt\mu}  \sum_{i} \sigma_{i} V_i(X_t)\circ\di W^i_t
-\frac{\varkappa^2}{2\mu} \Big(
\sum_{i} \lambda_i^2  D V_i [ V_i](X_t)
+\sum_{j< i} \mu_i\mu_j  [ V_i,V_j](X_t)
\Big)\,\di t.
\ea
\end{thm}

\begin{thm}
\label{t:B}
Suppose that assumptions
$\mathbf{A}_V$,
$\mathbf{A}_\tau$,
$\mathbf{A}_{V,\mu}$ and
$\mathbf{A}_\varkappa$ hold. Then, for any initial values $y\in\bR^d$ and $\iota\in \{1,\dots,p\}$,
the gliding evolutions $\{Y^{\e,\delta}\}$ converge weakly in $C(\bR_+,\bR^d)$
to a diffusion $Y$ starting at $y$ with the generator
\ba
\label{e:opB}
Bf
&=\frac{\varkappa^2}{2\mu} \sum_{i}  \sigma_{i}^2  \langle \nabla^2 f V_i ,V_i\rangle
+ \frac{\varkappa^2}{2\mu}\sum_{i} \lambda_{i}^2 \langle \nabla f, DV_i[V_i]\rangle
+\frac{\varkappa^2}{\mu}\sum_{j<i} \mu_i\mu_j \langle \nabla f, DV_i[V_j] \rangle
,\quad f\in C^3_b(\bR^d,\bR).
\ea
The diffusion $Y$ is a solution to an It\^o-SDE
\ba
\label{e:SDEY}
\di Y_t= \frac{\varkappa}{\sqrt\mu}  \sum_{i} \sigma_{i} V_i(Y_t)\,\di W^i_t +
\frac{\varkappa^2}{\mu} \Big(
\frac12\sum_{i} \lambda_{i}^2 DV_i[V_i](Y_t)
+\sum_{j< i} \mu_i\mu_j  D V_i [ V_j](Y_t)
\Big)\,\di t,
\ea
where $W^1,\dots,W^p$ are independent standard one-dimensional Brownian motions.
The SDE \eqref{e:SDEY} can be rewritten in the Stratonovich form
as
\ba
\label{e:SDEYStrat}
\di Y_t= \frac{\varkappa}{\sqrt\mu}  \sum_{i} \sigma_{i} V_i(Y_t)\circ\di W^i_t
+\frac{\varkappa^2}{2\mu} \sum_{j< i} \mu_i\mu_j  [ V_i,V_j](Y_t)\,\di t.
\ea
\end{thm}

The two limiting diffusions $X$ and $Y$ have the same second-order part, determined by the fluctuations of the run times,
but generally different drifts. The cross terms $DV_i[V_j]$, $j<i$, reflect the prescribed ordering of the vector
fields within one cycle.
In Stratonovich form, the order-dependent part of the drift is expressed through their Lie brackets.

We illustrate our findings by several examples and start with a generic example of run and tumble motion along two non-commutative
vector fields.

\begin{exa}[cyclic motion along two vector fields in $\bR^d$]
\label{exa:UV}
Let $(T_k)_{k\in\mathbb N_0}$ be the arrival times of a Poisson process with intensity $c\in (0,\infty)$.
Then, $ T_{k+1}-T_k\stackrel{\di}{=} \mathrm{Exp}(c)$, $\mu_i=1/c$, $\lambda_i^2=2/c^2$,
$\sigma_i^2=1/c^2$, $i=1,2$.
Let $V,U\colon\bR^d\to\bR^d$ be two vector fields. Then, any permutation of the vector fields
$\{U,V,-U,-V\}$ defines a four-step cyclic motion in $\bR^d$, $p=4$.
Since a cyclic shift corresponds only to a change of the initial direction
and does not affect the limiting diffusion, the $4!=24$ permutations reduce
to six inequivalent cyclic orderings:
\ba
(V^{(1)}_1,V^{(1)}_2,V^{(1)}_3,V^{(1)}_4)&=(U, -U, V, -V),\\
(V^{(2)}_1,V^{(2)}_2,V^{(2)}_3,V^{(2)}_4)&=(U, -U, -V, V),\\
(V^{(3)}_1,V^{(3)}_2,V^{(3)}_3,V^{(3)}_4)&=(U, V, -U, -V),\\
(V^{(4)}_1,V^{(4)}_2,V^{(4)}_3,V^{(4)}_4)&=(U, V, -V, -U),\\
(V^{(5)}_1,V^{(5)}_2,V^{(5)}_3,V^{(5)}_4)&=(U, -V, -U, V),\\
(V^{(6)}_1,V^{(6)}_2,V^{(6)}_3,V^{(6)}_4)&=(U, -V, V, -U).
\ea
These orders are schematically represented in Fig.\ \ref{f:two}.
In particular, the sequences $(3)$ and $(5)$ encode the counterclockwise and clockwise motions, respectively.
The limiting operators $A$ and $B$ have the following form:
\ba
A^{(1)}f&=A^{(2)}f=A^{(4)}f =A^{(6)}f\\
&= \frac{\varkappa^2}{4c} \Big( \langle \nabla^2 f  V, V\rangle +\langle \nabla^2 f  U, U\rangle
- \langle \nabla f , D U [ U] + D V [ V]\rangle\Big),\\
A^{(3)}f&=\frac{\varkappa^2}{4c} \Big(
\langle \nabla^2 f  V, V\rangle +\langle \nabla^2 f  U, U\rangle
- \langle \nabla f , D U [ U] + D V [ V] -[U,V] \rangle\Big),\\
A^{(5)}f&=\frac{\varkappa^2}{4c} \Big(
\langle \nabla^2 f  V, V\rangle +\langle \nabla^2 f  U, U\rangle
- \langle \nabla f , D U [ U] + D V [ V] +[U,V] \rangle\Big),\\
\ea
and
\ba
B^{(1)}f&=B^{(2)}f=B^{(4)}f =B^{(6)}f\\
&= \frac{\varkappa^2}{4c} \Big( \langle \nabla^2 f  V, V\rangle +\langle \nabla^2 f  U, U\rangle
+\langle \nabla f , D U [ U] + D V [ V]\rangle\Big),\\
B^{(3)}f&=\frac{\varkappa^2}{4c} \Big( \langle \nabla^2 f  V, V\rangle +\langle \nabla^2 f  U, U\rangle
+\langle \nabla f , D U [ U] + D V [ V] +[U,V]\rangle\Big),\\
B^{(5)}f
&= \frac{\varkappa^2}{4c} \Big( \langle \nabla^2 f  V, V\rangle +\langle \nabla^2 f  U, U\rangle
+ \langle \nabla f , D U [ U] + D V [ V]-[U,V]\rangle
\Big).
\ea
In particular, the orderings (3) and (5), corresponding to opposite orientations of the cycle,
produce Lie-bracket drift terms with opposite signs.
Thus, the orientation of the microscopic cyclic motion remains visible in the macroscopic diffusion limit.

\begin{figure}
\begin{center}
\includegraphics{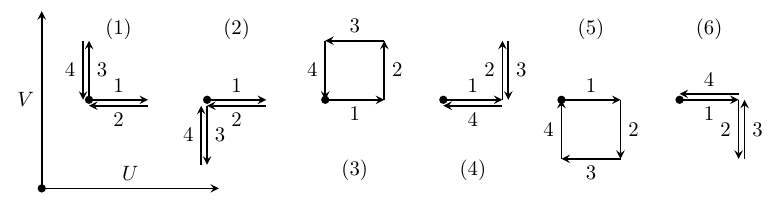}
\end{center}
\caption{Six inequivalent cyclic motions along vector fields $\pm U$ and $\pm V$.\label{f:two}}
\end{figure}
\end{exa}

\begin{exa}[homogeneous velocities]
Assume that all $V_1,\dots,V_p$ are constant vector fields. Then,
\ba
Af(x)=Bf(x)
=\frac{\varkappa^2}{2\mu}  \sum_{i} \sigma_{i}^2 \langle \nabla^2 f  V_i, V_i\rangle,
\ea
and the limit processes $X$ and $Y$ are a $d$-dimensional Brownian motion with the (possibly degenerate) covariance matrix
\ba
\Sigma &= \frac{\varkappa^2}{\mu}  \sum_{i} \sigma_{i}^2  V_i V_i^T,
\ea
that is,
\ba
X_t=Y_t&=x+\Sigma^{1/2} W_t,
\ea
where $W$ is a standard $d$-dimensional Brownian motion. Alternatively,
\ba
X_t=Y_t= x+\frac{\varkappa}{\sqrt\mu}  \sum_{i} \sigma_{i} V_i W^i_t,
\ea
where $W^1,\dots,W^p$ are independent standard one-dimensional Brownian motions.
\end{exa}

\begin{exa}[commuting vector fields]
Assume that the vector fields $V_1,\dots, V_p$ commute, i.e.,
\ba
{}[V_i,V_j]=DV_j[V_i]-DV_i[V_j]\equiv 0,\quad i,j=1,\dots, p,
\ea
so that the order-dependent Lie-bracket contribution in \eqref{e:opA} and \eqref{e:opB} vanishes.
Then, in view of the assumption $\mathbf{A}_{V,\mu}$
\ba
\sum_{j<i}\mu_i\mu_j DV_i[V_j]=-\frac12 \sum_{i}\mu_i^2 DV_i[V_i]
\ea
and
the operators $A$ and $B$ take the following simpler form:
\ba
Af
&= \frac{\varkappa^2}{2\mu}  \sum_{i} \sigma_{i}^2 \langle \nabla^2 f  V_i, V_i\rangle
- \frac{\varkappa^2}{2\mu} \sum_{i} \mu_i^2 \langle \nabla f , D V_i [ V_i] \rangle,\\
Bf
&=\frac{\varkappa^2}{2\mu} \sum_{i}  \sigma_{i}^2  \langle \nabla^2 f V_i ,V_i\rangle
+ \frac{\varkappa^2}{2\mu} \sum_{i}  \sigma_{i}^2 \langle \nabla f , D V_i [ V_i] \rangle
 ,\quad f\in C^3_b(\bR^d,\bR).
\ea
Consequently, the limiting diffusions satisfy the following SDEs:
\ba
\di X_t&=
\frac{\varkappa}{\sqrt\mu} \sum_{i}  \sigma_{i}  V_i (X_t)\, \di W^i_t
- \frac{\varkappa^2}{2\mu} \sum_{i} \mu_i^2 D V_i [ V_i](X_t)\,\di t\\
&= \frac{\varkappa}{\sqrt\mu} \sum_{i}  \sigma_{i}  V_i (X_t)\circ \di W^i_t
- \frac{\varkappa^2}{2\mu} \sum_{i} \lambda_i^2 D V_i [ V_i](X_t)\,\di t,\\
\di Y_t& = \frac{\varkappa}{\sqrt\mu} \sum_{i}  \sigma_{i}  V_i (Y_t)\, \di W^i_t
+\frac{\varkappa^2}{2\mu} \sum_{i} \sigma_i^2 D V_i [ V_i](Y_t)\,\di t\\
&= \frac{\varkappa}{\sqrt\mu} \sum_{i}  \sigma_{i}  V_i (Y_t)\circ \di W^i_t.
\ea
\end{exa}

\begin{exa}[run-and-reverse motion in $\bR^d$]
\label{exa:rr}
Assume that $p=2$, $V\colon \bR^d\to\bR^d$ and
\ba
V_1=V,\quad V_2=-V.
\ea
Let $(T_k)_{k\in\mathbb N_0}$ be the arrival times of a Poisson process with intensity $c\in (0,\infty)$ as in Example \ref{exa:UV}.
The vector fields $V_1$ and $V_2$ obviously commute, and the generators $A$ and $B$ take the form
\ba
Af
&= \frac{\varkappa^2}{2c}  \langle \nabla^2 f  V, V\rangle
- \frac{\varkappa^2}{2c}\langle \nabla f , D V [ V] \rangle,\\
Bf
&=\frac{\varkappa^2}{2c} \langle \nabla^2 f V ,V\rangle
+ \frac{\varkappa^2}{2c}  \langle \nabla f , D V [ V] \rangle\\
&=\frac{\varkappa^2}{2c}  \langle V , \nabla \langle V,\nabla f\rangle\rangle
 ,\quad f\in C^3_b(\bR^d,\bR).
\ea
The limiting dynamics is degenerate and is governed by a one-dimensional Brownian motion $W$.
In the gliding regime, the diffusion $Y$ occurs along the integral curves of the vector field $V$,
\ba
\di Y_t&= \frac{\varkappa}{\sqrt c} V (Y_t)\circ \di W_t.
\ea
In the flight regime,
\ba
\di X_t&=
\frac{\varkappa}{\sqrt c}  V (X_t)\circ \di W_t - \frac{\varkappa^2}{c} D V [ V](X_t)\,\di t.
\ea
Thus, the diffusion component of $X$ is tangent to the one-dimensional foliation generated by $V$,
whereas the additional drift acts in the direction $-DV[V]$.

In dimension $d=1$, the limiting generator
$Bf=\frac{\varkappa^2}{2c} V(Vf')'$ was derived in \cite[Theorem 5.1]{ratanov1999telegraph}.
\end{exa}

\begin{exa}[planar circular run-and-reverse motion]
In the setting of Example \ref{exa:rr},
assume that $d=2$, $x=(x^1,x^2)$, $v\in (0,\infty)$, $\theta\in [0,\pi)$, and
\ba
V(x)=V(x;\theta)=\frac{v}{\|x\|}\begin{pmatrix}
                                  x^1\sin\theta-x^2 \cos \theta\\
                                  x^1 \cos \theta + x^2\sin\theta
                                 \end{pmatrix}
= \frac{v}{\|x\|}\begin{pmatrix}
                                  \sin\theta &-\cos \theta\\
                                  \cos \theta  &\sin\theta
                                 \end{pmatrix}         x
                                 ,\quad
V_1=V,\quad V_2=-V.
\ea
Since the vector field $V$ is singular at the origin, throughout this example we consider
the dynamics away from $x=0$, where $V$ is smooth and bounded.

The deterministic dynamics along the vector field $V$ looks as follows.
For $\theta=0$ this is a purely tangential (circular) counterclockwise motion,
for $\theta\in (0,\pi/2)$ the trajectories are outward spirals, for  $\theta=\pi/2$, the motion is radially outward,
for $\theta\in (\pi/2,\pi)$ the trajectories are inward spirals. The angles $\theta\in [\pi,2\pi)$ correspond to the
motion in an opposite direction, $V(x;\theta+\pi)=-V(x;\theta)$.

The directional derivative $DV[V]$ has the form
\ba
DV[V](x)=\frac{v^2}{\|x\|^2}\begin{pmatrix}
                                  -x^1\cos^2\theta -x^2 \sin \theta\cos \theta\\
                                  - x^2 \cos^2 \theta + x^1 \sin \theta \cos \theta
                                 \end{pmatrix}
=\frac{v^2\cos\theta}{\|x\|^2}\begin{pmatrix}
                                  -\cos\theta & -\sin \theta \\
                                  \sin \theta &- \cos \theta
                                 \end{pmatrix}x.
\ea
It is easy to see that $V$ and $DV[V]$ are orthogonal, so that whereas the gliding dynamics gives rise to
a one-dimensional diffusion along an integral curve of $V$, the flight dynamics $X$ acquires an additional radial drift
with the absolute value
\ba
\frac{\varkappa^2}{c} \|DV[V](x)\|= \frac{\varkappa^2}{c} \frac{v^2}{\|x\|}|\cos\theta|.
\ea
Sample trajectories of the diffusion approximations of these motions for different angles $\theta$ are presented in Fig.~\ref{f:planar}.

\begin{figure}
\includegraphics{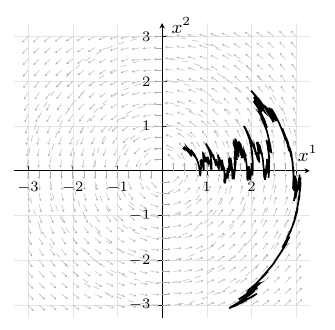}
\includegraphics{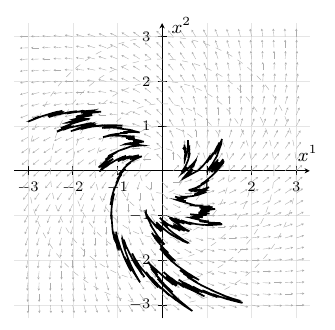}
\includegraphics{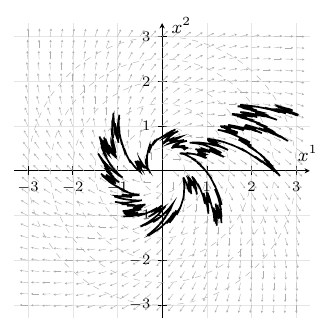}
\caption{Sample paths of the planar limiting diffusion $X$ (flights regime) for $c=1$, $v=1$, $x=(0.5,0.5)$ and
$\theta=0$ (l.), $\theta=\pi/4$ (c.), $\theta=3\pi/4$ (r.). The arrows indicate the direction of the vector field $V$.\label{f:planar}}
\end{figure}
\end{exa}

The following sections \ref{s:Awrc}--\ref{s:Aconv} will be devoted to
the proof of Theorem \ref{t:A}. Theorem \ref{t:B} will be proved in
section \ref{s:gliding}.
In all the arguments of these sections we will assume that assumptions
$\mathbf{A}_V$,
$\mathbf{A}_\tau$,
$\mathbf{A}_{V,\mu}$ and
$\mathbf{A}_\varkappa$ hold, without repeating it.
The proof will be organised as follows.
First we note that it is more convenient to work with the c\`adl\`ag version of the process $X^{\e,\delta}$,
namely, with the process $Z^{\e,\delta}$ defined as
\ba
Z^{\e,\delta}_t=X^{\e,\delta}_{T_k^\e}\quad \text{ on } t\in [T_k^\e,T^\e_{k+1}),\quad k\in\mathbb N_0.
\ea
It is clear that the values of these processes at arrival times coincide:
\ba
Z^{\e,\delta}_{T^\e_k}=X^{\e,\delta}_{T_k^\e},\quad k\in \mathbb N_0.
\ea
In section \ref{s:Awrc} we establish the weak relative compactness of the family $\{Z^{\e,\delta}\}$ in the Skorokhod space
$D([0,\infty),\bR^d)$ and show that the difference  $Z^{\e,\delta}- X^{\e,\delta}$ converges to zero in u.c.p.
Then, in section \ref{s:operator} we determine the limit of the $p$-step
pseudo-characteristic operator of $Z^{\e,\delta}$ (or, equivalently, $X^{\e,\delta}$). In section \ref{s:Aconv}
we show that any condensation point of the family $Z^{\e,\delta}$ is a solution
to a (well-posed) martingale problem associated with the operator $A$.
Finally, in section \ref{s:gliding} we discuss the necessary adjustments of the arguments that lead to the proof of
Theorem \ref{t:B}.

\section{Weak relative compactness of $\{Z^{\e,\delta}\}$\label{s:Awrc}}

The weak relative compactness of the family $\{Z^{\e,\delta}\}$ will follow from the
compact containment condition (Lemma \ref{l:cc}) and Aldous' criterion (Lemma \ref{l:aldous}).

For the proof, we will need the \emph{first passage process}
$N^\e=(N^\e_t)_{t\in[0,\infty)}$ related to the sequence of
arrival times $(T^\e_k)_{k\in\bN_0}$ defined as
\ba
N^\e_t&=\min\{k\in\mathbb N\colon T_k^\e> t\}.
\ea
The random variable $N^\e_t$ denotes the index of the first arrival after time $t$.
The important property of the process $N^\e$ is that for each $t\in[0,\infty)$
the random variable $N_t^\e$ is a stopping time with respect to the filtration
$\rF^\e_k:=\sigma(T_0^\e,\dots, T_k^\e)$, $k\in\mathbb N_0$,
since  for each $k\in\mathbb N$
\ba
\{N_t^\e \leq k\} & =\{T^\e_k > t \} \in \rF^\e_k,\quad k\in\mathbb N_0.
\ea
Moreover, we have useful equalities
\ba
\label{e:tnu}
\{N_t^\e=k\}& =\{T^\e_{k-1}\leq t< T^\e_k \}\in \rF^\e_k,\quad k\in\mathbb N,\\
\{N_t^\e\geq k\} & =\{T^\e_{k-1} \leq t \},\quad k\in\mathbb N.
\ea
In particular, the process $N^\e$ will be used to estimate the number of arrivals on a finite interval $[0,T]$ as a function of $\e$.
First, we prove the following lemma.

\begin{lem}[estimate on $N^\e$]
\label{l:Etnu}
There is $K\in (0,\infty)$ such that for any $t\in[0,\infty)$, any $\iota=1,\dots, p$, and any $\e\in(0,1]$
\ba
\E_\iota N^\e_t\leq \frac{t+K\e}{\mu_* \e},
\ea
where $\mu_*=\min\{\mu_1,\dots,\mu_p\}$.
\end{lem}
\begin{proof}
First we note that
\ba
T^\e_{N^\e_t} =t + ( T^\e_{N^\e_t}- t)
=t + R^\e_t,
\ea
where the value
\ba
R^\e_t=T^\e_{N_t^\e}-t
\ea
is the overshoot (or the residual lifetime) of the level $t$.

To estimate $\E_\iota R^\e_t$, we use Theorem 1.1 from \cite{spouge2007inequalities}, and,
in particular, the estimate given in Eq.\ (5) there.
Note that each random variable $\tau^\e_i=\e\tau_i$, $i=1,\dots, p$, is
stochastically dominated by the random variable $\e\widehat \tau$, where
$\widehat \tau=\tau_1+\cdots+\tau_{p}$.
Then, Eq.\ (5) in \cite{spouge2007inequalities} yields that
\ba
\label{e:R}
\sup_{t\in[0,\infty)}\E_\iota R^\e_t &\leq \frac12 \inf_{a\in(0,\infty)}\frac{\E (\widehat \tau^\e + a)^2}{a \P(\widehat \tau^\e\geq a) }
&=\frac12 \inf_{a\in(0,\infty)}\frac{\e^2 \E (\widehat \tau + a)^2}{\e a \P(\widehat \tau \geq a) }
=\e K
\ea
for some $K\in(0,\infty)$.
Therefore,
\ba
\E_\iota T^\e_{N^\e_t} &\leq t + \e K.
\ea
Furthermore, for $k\in\mathbb N_0$ we write
\ba
T_k^\e &= \sum_{j=1}^k (T_j^\e - T_{j-1}^\e) - \sum_{j=1}^k \E_\iota (T_j^\e - T_{j-1}^\e) + \sum_{j=1}^k \E_\iota (T_j^\e - T_{j-1}^\e) \\
&= M_k^\e + \sum_{j=1}^k \E_\iota (T_j^\e - T_{j-1}^\e) \\
&\geq  M_k^\e + k \e \mu_*,
\ea
where $M^\e$ is a martingale w.r.t.\ the filtration $(\rF_k^\e)_{k\in\mathbb N_0}$
and $\mu_*=\min\{\mu_1,\dots,\mu_p\}\in (0,\infty)$.

Hence, for any $n\in\bN$
\ba
T_{N_t^\e}^\e
&\geq T_{N_t^\e\wedge n}^\e \geq  M_{N_t^\e\wedge n}^\e +\e\mu_* (N_t^\e\wedge n).
\ea
Since $N_t^\e\wedge n$ is a bounded stopping time, the optional stopping theorem gives
$\E_\iota M_{N_t^\e\wedge n}^\e=0$
and therefore,
\ba
\E_\iota (N_t^\e\wedge n) & \leq \frac{t  +\e K}{\e \mu_*}.
\ea
Applying the monotone convergence theorem finishes the proof.
\end{proof}

Now we are able to prove the compact containment condition (see, e.g., \cite[Chapter 3, Theorem 9.1]{EthierK-86}).

\begin{lem}[compact containment]
\label{l:cc}
For any $T\in[0,\infty)$ and $\eta\in (0,1]$, there is
$R=R_{T,\eta}\in(0,\infty)$ such that
\ba
\liminf_{\e,\delta\downarrow 0} \inf_{x\in\bR^d,\, \iota\in\{1,\dots,p\}}
\P_{x,\iota}\Big(\sup_{t\in[0,T]}|Z^{\e,\delta}_t-x|\leq  R\Big) \geq  1 -\eta.
\ea
\end{lem}
\begin{proof}
Let $T\in(0,\infty)$.
Since $Z^{\e,\delta}$ is constant on each interval between successive
arrival times, it is sufficient to consider its values at the arrival times and
to show that
\ba
\liminf_{\e,\delta\downarrow 0}
 \inf_{x\in\bR^d,\, \iota\in\{1,\dots,p\}}
\P_{x,\iota}\Big(\sup_{k\leq N^\e_T }|Z^{\e,\delta}_{T_k^\e}-x|\leq  R\Big) \geq  1 -\eta.
\ea
Consider the martingale difference sequence
\ba
\Delta_0^{\e,\delta}&:=0,\\
\Delta_{k}^{\e,\delta}& = Z^{\e,\delta}_{T_{k}^\e}- Z^{\e,\delta}_{T_{k-1}^\e}
- \E \Big[ Z^{\e,\delta}_{T^\e_{k}}- Z^{\e,\delta}_{T^\e_{k-1}}\Big|\rF^{\e,\delta}_{T^\e_{k-1}}  \Big]\\
&=\frac{1}{\delta}V_{n_{k-1}}(Z^{\e,\delta}_{T^\e_{k-1}})  \Big[ T^\e_{k}- T^\e_{k-1}- \E_\iota  (T^\e_{k}-T^\e_{k-1})\Big],
\quad k\in\mathbb N,
\ea
and the associated martingale
\ba
M_k^{\e,\delta}:=\sum_{j=1}^k \Delta^{\e,\delta}_{j}, \quad k\in\mathbb N, \quad M_0^{\e,\delta}=0.
\ea
Then,
\ba
Z^{\e,\delta}_{T^\e_k}-x
&=M_k^{\e,\delta} + A_k^{\e,\delta},
\ea
where
\ba
A_k^{\e,\delta}:=\sum_{j=1}^k \E \Big[ Z^{\e,\delta}_{T^\e_j}- Z^{\e,\delta}_{T^\e_{j-1}}\Big|\rF^{\e,\delta}_{T^\e_{j-1}}  \Big],
\quad k\in\mathbb N,
\quad A_0^{\e,\delta}=0.
\ea
Consequently,
\ba
\label{e:e1}
\P_{x,\iota}\Big(\max_{t\in[0,T]}|Z^{\e,\delta}_t-x|> R\Big)
&\leq \P_{x,\iota}\Big(\max_{k\leq N^\e_T }|M_k^{\e,\delta}|+\max_{k\leq N^\e_T }|A_k^{\e,\delta}|> R\Big)\\
&\leq \P_{x,\iota}\Big(\max_{k\leq N^\e_T }|M_k^{\e,\delta}|>\frac{R}{2}\Big)
+\P_{x,\iota}\Big(\max_{k\leq N^\e_T }|A_k^{\e,\delta}|> \frac{R}{2}\Big).
\ea
By Lemma \ref{l:Etnu} and Markov's inequality there is $N\in (0,\infty)$ large enough and $K\in (0,\infty)$ such that
\ba
\label{e:NT}
\P_\iota\Big(N_T^\e> \frac{N T}{\e} \Big)
\leq \frac{\e}{NT} \E_\iota N_T^\e
\leq \frac{\e}{NT}\frac{T+K\e}{\mu_*\e}
<\frac{\eta}{3}
\ea
for all $\e\in(0,1]$.
Hence,
\ba
\P_{x,\iota}\Big(\max_{k\leq N^\e_T }|M_k^{\e,\delta}|>\frac{R}{2}\Big)
\leq \P_{x,\iota}\Big(\max_{k\leq NT/\e}|M_k^{\e,\delta}|>\frac{R}{2}\Big)
+ \frac{\eta}{3}.
\ea
With the help of Doob's maximal inequality we get
\ba
\label{e:e2}
 \P_{x,\iota}\Big( \max_{k\leq  NT/\e} |M^{\e,\delta}_k|> \frac{R}{2}\Big)
&\leq \frac{4}{R^2}\E_{x,\iota} |M^{\e,\delta}_{\lfloor NT/\e\rfloor}|^2\\
&\leq  \frac{4}{R^2}\sum_{k=1}^{\lfloor NT/\e\rfloor}\E_{x,\iota} |\Delta^{\e,\delta}_k|^2\\
&\leq \frac{4}{R^2} \frac{\|V\|^2}{\delta^2}\frac{NT}{\e}\e^2 \max_{1\leq i\leq p} \Var\tau_i
<\frac{\eta}{3}
\ea
for $\e,\delta$ small and $R$ large enough.

We treat now the terms $A_k^{\e,\delta}$.
First, we note that a straightforward estimate of the absolute value of $A^{\e,\delta}$ does not work.
Indeed, it is easy to see that for
each $k\in\mathbb N_0$,
$|A^{\e,\delta}_k|\leq k\frac{\e}{\delta}\|V\|\max\{\mu_1,\dots,\mu\} $. Since the number of arrivals on the interval $[0,T]$
must be of order $\frac{1}{\e}$, a simple estimate yields that $|A^{\e,\delta}_{N^\e_T}|$ is of order $\frac{1}{\delta}$
which is not helpful.

To obtain a useful estimate, we group the increments into complete cycles of length $p$ and use the centering assumption
\eqref{e:Vmu}. The leading contribution of each complete cycle then vanishes, leaving only an error caused by
the spatial variation of the vector fields.

First, we consider the term $A^{\e,\delta}$ at time instants with indices $k$ which are multiples of $p$, i.e., assume first that
$k=mp$ for some $m\in\mathbb N$.
Then,
\ba
A_{mp}^{\e,\delta}
=\sum_{j=1}^{mp}
\E \Big[ Z^{\e,\delta}_{T^\e_j}- Z^{\e,\delta}_{T^\e_{j-1}}\Big|\rF^{\e,\delta}_{T^\e_{j-1}}  \Big]
&=\frac{\e }{\delta}\sum_{i=1}^{m} \sum_{j=1}^{p}
\mu_{n_{j-1}}V_{n_{j-1}}\Big(Z^{\e,\delta}_{T^\e_{(i-1)p+j-1}} \Big).
\ea
By assumption $\mathbf{A}_{V,\mu}$, we rewrite
\ba
\sum_{j=1}^{p}
\mu_{n_{j-1}}V_{n_{j-1}}(Z^{\e,\delta}_{T^\e_{(i-1)p+j-1}} )
&=
\sum_{j=1}^{p} \mu_{n_{j-1}}V_{n_{j-1}} (Z^{\e,\delta}_{T^\e_{(i-1)p}} )\quad (=0)\\
&+ \sum_{j=2}^{p}
\mu_{n_{j-1}} \Big(V_{n_{j-1}}(Z^{\e,\delta}_{T^\e_{(i-1)p+j-1}} ) - V_{n_{j-1}}(  Z^{\e,\delta}_{T^\e_{(i-1)p}})\Big).
\ea
Therefore,
\ba
\Big|\sum_{j=1}^{p}
\mu_{n_{j-1}} V_{n_{j-1}}(Z^{\e,\delta}_{T^\e_{(i-1)p+j-1}} )\Big|
&\leq\sum_{j=2}^{p} \mu_{n_{j-1}} \|\nabla V_{n_{j-1}}\|
\Big|Z^{\e,\delta}_{T^\e_{(i-1)p+j-1}} - Z^{\e,\delta}_{T^\e_{(i-1)p}}\Big| .
\ea
Furthermore, for each $j=2,\dots, p$,
\ba
\Big|Z^{\e,\delta}_{T^\e_{(i-1)p+j-1}} - Z^{\e,\delta}_{T^\e_{(i-1)p}}\Big|
&\leq  \| V\|
\frac{1}{\delta} \Big(T^\e_{(i-1)p+j-1} - T^\e_{(i-1)p}\Big)
\leq_C  \frac{1}{\delta} (T^\e_{ip} - T^\e_{(i-1)p}).
\ea
Finally, this gives
\ba
\label{e:Amp}
|A_{mp}^{\e,\delta}|&=\Big|\sum_{j=1}^{mp}
\E \Big[ Z^{\e,\delta}_{T^\e_j}- Z^{\e,\delta}_{T^\e_{j-1}}\Big|\rF^{\e,\delta}_{T^\e_{j-1}}  \Big]\Big|\\
&\leq_C \frac{\e}{\delta^2} \sum_{i=1}^{m} (T^\e_{ip} - T^\e_{(i-1)p})\\
&\leq_C  T_{mp}^\e.
\ea
For intermediate
indices $k=mp + j$, $j=1,\dots,p-1$ we combine \eqref{e:Amp} with the straightforward estimate:
\ba
|A_{mp+j}^{\e,\delta}|
&\leq |A_{mp}^{\e,\delta}|
+\sum_{l=mp+1}^{mp+j} \Big|\E \Big[ Z^{\e,\delta}_{T^\e_l}- Z^{\e,\delta}_{T^\e_{l-1}}\Big|\rF^{\e,\delta}_{T^\e_{l-1}}  \Big]\Big|\\
&\leq
|A_{mp}^{\e,\delta}|
+ \sum_{l=1}^{j}\Big|
\frac{\e \mu_{n_{l-1}}}{\delta}V_{n_{l-1}}(Z^{\e,\delta}_{T^\e_{mp+l-1}} )\Big|\\
&\leq |A_{mp}^{\e,\delta}| +  \mu\|V\| \frac{\e}{\delta}\\
&\leq_C T_{mp}^\e + \frac{\e}{\delta}\\
&\leq_C T_{k}^\e +\sqrt{\e}.
\ea
Overall, there is $C\in (0,\infty)$ such that for each $k\in\mathbb N$
\ba
\label{e:e3}
|A_{k}^{\e,\delta}|\leq C( T_k^\e + \sqrt{\e})
\ea
and
\ba
\label{e:A}
\P_{x,\iota}\Big(\max_{k\leq N^\e_T }|A_k^{\e,\delta}|> \frac{R}{2}\Big)
&\leq \P_\iota\Big(T_{N_T^\e} +\sqrt{\e}> \frac{R}{2C}\Big)\\
&=\P_\iota\Big(T_{N_T^\e}-T + T +\sqrt{\e}> \frac{R}{2C}\Big)\\
&\leq \P_\iota\Big(R^\e_T > \frac{R}{4C}\Big) + \P\Big(T + \sqrt{\e}> \frac{R}{4C}\Big)\\
&\leq_C \frac{\e}{R}+0\leq \frac{\eta}{3}
\ea
for $R$ large, where for the estimate of the overshoot $R^\e_T$ we used \eqref{e:R}.
\end{proof}

To prove the Aldous' criterion (see, e.g.,  \cite[\S 4a, Chapter VI]{JacodS-03}) we denote by $\mathbb F^{\e,\delta}$
the natural filtration of the process $Z^{\e,\delta}$.

\begin{lem}[Aldous' criterion]
\label{l:aldous}
For any $\eta\in(0,1]$, any $T\in[0,\infty)$
\ba
\lim_{h\downarrow 0}\limsup_{\e,\delta\downarrow 0}
\sup_{\sigma,\tau\in\mathbb F^{\e,\delta},
\sigma\leq\tau\leq(\sigma+h)\wedge T}
\P_{x,\iota}\Big(|Z^{\e,\delta}_\tau-Z^{\e,\delta}_\sigma|\geq\eta\Big)=0.
\ea
The convergence is uniform over
$(x,\iota)\in\mathbb R^d\times\{1,\dots,p\}$.
\end{lem}

\begin{proof}
Since $Z^{\e,\delta}$ is piecewise constant, it is enough to consider stopping times
$\sigma,\tau\in\{T_k^\e\}_{k\in\mathbb N_0}$.
Recall that the process
$k\mapsto (Z^{\e,\delta}_{T_k^\e},n_k)$
is a Markov chain with respect to the filtration
$(\rF^{\e,\delta}_{T_k^\e})_{k\in\mathbb N_0}$.
Hence, for any $\eta\in(0,1]$,
\ba
\sup_{\sigma\leq\tau\le\sigma+h}
\P_{x,\iota}\Big(|Z^{\e,\delta}_\tau-Z^{\e,\delta}_\sigma|\geq\eta\Big)
&=
\sup_{\substack{T^\e_{j+k}-T^\e_j\leq h}}
\P_{x,\iota}\Big(
|Z^{\e,\delta}_{T^\e_{j+k}}-Z^{\e,\delta}_{T^\e_j}|
\geq\eta\Big) \\
&\leq
\sup_{T^\e_{j+k}-T^\e_j\leq h}
\E_{x,\iota}\Big[
\P\Big(
|Z^{\e,\delta}_{T^\e_{j+k}}-Z^{\e,\delta}_{T^\e_j}|
\geq\eta \Big| \rF^{\e,\delta}_{T^\e_j} \Big)\Big] \\
&\leq
\sup_{x\in\mathbb R^d,\ 1\leq\iota\le p}
\P_{x,\iota}\Big( \sup_{t\in[0,h]}|Z^{\e,\delta}_t-x| \geq\eta \Big).
\ea
Now we argue exactly as in Lemma~\ref{l:cc}, replacing $R$ by $\eta$ and $T$ by $h$. This yields the analogue of \eqref{e:e1}:
\ba
\P_{x,\iota}\Big(
\sup_{t\in[0,h]}|Z^{\e,\delta}_t-x|
\geq\eta
\Big)
&\leq
\P_{x,\iota}\Big( \max_{k\le N_h^\e}|M_k^{\e,\delta}| \geq\frac{\eta}{2} \Big)
+
\P_{x,\iota}\Big( \max_{k\le N_h^\e}|A_k^{\e,\delta}|
\ge\frac{\eta}{2} \Big).
\ea
Let $\theta\in(0,1]$.
The estimate \eqref{e:NT} takes the form
\ba
\limsup_{\e,\delta\downarrow 0}
\P_\iota\Big(N_h^\e> \frac{Nh}{\e} \Big)
\leq \limsup_{\e,\delta\downarrow 0}
\frac{\e}{Nh}\frac{h+K\e}{\mu_*\e}
=\frac{1}{N\mu_*}
<\frac{\theta}{3}
\ea
for $K=K(\theta)$ large enough.

The estimate \eqref{e:e2} takes the form
\ba
\limsup_{\e,\delta\downarrow 0} \P\Big( \max_{k\leq  Nh/\e} |M^{\e,\delta}_k|> \frac{\eta}{2}\Big)
&\leq \frac{4}{\eta^2}\E |M^{\e,\delta}_{\lfloor Nh/\e\rfloor}|^2\\
&\leq \frac{4\|V\|^2}{\eta^2\delta^2} \frac{Nh}{\e}\e^2 \max_{1\leq i\leq p} \Var\tau_i
\leq\frac{\theta}{3}
\ea
for $h$ small enough.

The estimate \eqref{e:A} takes the form
\ba
\limsup_{\e,\delta\downarrow 0}
\P_{x,\iota}\Big(\max_{k\leq N^\e_h }|A_k^{\e,\delta}|> \frac{\eta}{2}\Big)
&\leq \limsup_{\e,\delta\downarrow 0}\P_\iota\Big(T_{N_h^\e} +\sqrt{\e}> \frac{\eta}{2C}\Big)\\
&\leq \limsup_{\e,\delta\downarrow 0}\P_\iota\Big(R^\e_h > \frac{\eta}{4C}\Big)
+ \limsup_{\e,\delta\downarrow 0}\P\Big(h +\sqrt{\e}> \frac{\eta}{4C}\Big)\\
&=0
\ea
for $h$ small enough, where for the estimate of the overshoot $R^\e_h$ we used \eqref{e:R}.

Overall, these estimates imply that
\ba
\lim_{h\downarrow 0}\limsup_{\e,\delta\downarrow 0} \sup_{x,\iota}
\P_{x,\iota}\Big( \sup_{t\in[0,h]}|Z^{\e,\delta}_t-x| \geq\eta \Big)= 0.
\ea
This proves the claim.
\end{proof}

Finally we show that the piecewise linear process $X^{\e,\delta}$ and the associated piecewise constant process $Z^{\e,\delta}$
are close in the u.c.p.\ sense.

\begin{lem}
\label{l:ucp}
For any $T\in[0,\infty)$ and any $\eta\in(0,1]$
\ba
\lim_{\e,\delta \downarrow 0} \sup_{x\in\bR^d, \iota\in \{1,\dots, p\}}
\P_{x,\iota}\Big(\sup_{t\in[0,T]}|Z^{\e,\delta}_t - X^{\e,\delta}_t|\geq \eta\Big)=0.
\ea
\end{lem}
\begin{proof}
Since by \eqref{e:Xe} the process $t\mapsto X_t^{\e,\delta}$ is a linear
function on each interval $[T^\e_k,T^\e_{k+1}]$, $k\in\mathbb N_0$,
we have
\ba
\P_{x,\iota}\Big(\sup_{t\in[0,T]}|Z^{\e,\delta}_t - X^{\e,\delta}_t|\geq\eta\Big)
&\leq\P_{\iota}\Big(\frac{\|V\|}{\delta}\max_{k\leq N^\e_T}(T_k^\e-T_{k-1}^\e)\geq \eta\Big).
\ea
Since
\ba
\Big\{\max_{k\leq N^\e_T}(T^\e_k-T_{k-1}^\e)\geq \frac{\eta\delta}{\|V\|}\Big\}
=\bigcup_{k=1}^\infty\Big\{T^\e_k-T_{k-1}^\e \geq   \frac{\eta\delta}{\|V\|},k\leq N_T^\e\Big\} ,
\ea
with the help of \eqref{e:tnu} and independence of
$T^\e_k-T_{k-1}^\e$  and $T^\e_{k-1}$ for all $k\in\mathbb N$
we get
\ba
\label{e:enu}
\P_\iota\Big(\max_{k\leq N^\e_T}(T^\e_k-T_{k-1}^\e)\geq \frac{\eta\delta}{\|V\|}\Big)
&\leq \sum_{k=1}^\infty \P_\iota\Big(T^\e_k-T_{k-1}^\e \geq \frac{\eta\delta}{\|V\|},  T^\e_{k-1}\leq  T\Big) \\
&= \sum_{k=1}^\infty \P_\iota\Big(T^\e_k-T_{k-1}^\e \geq \frac{\eta\delta}{\|V\|}\Big)\P_\iota(T^\e_{k-1}\leq T) \\
&\leq \max_{1\leq i\leq p}\P\Big(\tau_i >\frac{\eta\delta}{\e \|V\|}\Big)\sum_{k=1}^\infty \P_\iota(T^\e_{k-1}\leq T)\\
&=\max_{1\leq i\leq p}\P\Big(\tau_i >\frac{\eta\delta}{\e \|V\|}\Big)\E_\iota N_T^\e.
\ea
We use the $(2+\rho)$-integrability of $\tau_i$, Markov's inequality, Lemma \ref{l:Etnu}
and assumption $\mathbf{A}_\varkappa$
to get
\ba
\P\Big(\max_{k\leq N^\e_T}(T^\e_k-T_{k-1}^\e)\geq \frac{\eta\delta}{\|V\|}\Big)
&\leq
\Big(\frac{\e \|V\|}{\eta\delta}\Big)^{2+\rho}\max_{1\leq i\leq p}\E\tau_i^{2+\rho}
\frac{T+\e K}{\mu_*\e}\\
&\leq_C \frac{1}{\e}\Big(\frac{\e}{\delta}\Big)^{2+\rho}
= \Big(\frac{\e}{\delta^2}\Big)^{\frac{2+\rho}{2}} \e^{\rho/2}
\leq_C \e^{\rho/2}
\to 0,\quad \e,\delta\downarrow 0.
\ea
\end{proof}

\section{Asymptotics of the pseudo-characteristic operator of $Z^{\e,\delta}$\label{s:operator}}

To identify the limiting dynamics of the family $\{Z^{\e,\delta}\}$, we study its evolution over one complete cycle.
More precisely, we consider the $p$-step pseudo-characteristic operator $A_{\iota}^{\e,\delta}$
\ba
A_\iota^{\e,\delta }f(x)&=\frac{\E_{x,\iota} f(Z^{\e,\delta}_{T^\e_{p}}) - f(x)}{\E_\iota T_p^\e }
\ea
defined for sufficiently smooth functions $f\colon \bR^d \to \bR$.
The use of a complete cycle is essential, since the leading first-order contribution then vanishes by the centering condition
$\mathbf{A}_{V,\mu}$.

\begin{lem}
\label{l:A}
For any $f\in C^3_b(\bR^d,\bR)$,
$x\in\bR^d$, any $\iota\in \{1,\dots,p\}$ we have
\ba
\label{e:Aiota}
\lim_{\e,\delta \downarrow 0} \sup_{x\in\mathbb R^d,\iota\in\{1,\dots,p\}}
|A_\iota^{\e,\delta }f(x) -A f(x)|=0,
\ea
where the operator $A$ is defined in \eqref{e:opA}.
\end{lem}

\begin{proof}
In order prove \eqref{e:Aiota} we first expand the contribution of the first-order Taylor terms of
$f(Z^{\e,\delta}_{T^\e_{p}})$
and use the centering condition to identify the order-dependent drift.
We then treat the second-order Taylor terms, which determine the diffusion coefficient.
Finally, we show that the resulting operator is invariant under cyclic shifts of the initial direction.

1. Let $\iota=1$.
For $i=1,\dots,p$,
denote $\Delta_i T^\e:=T_{i}^\e-T_{i-1}^\e\stackrel{\di}{=}\e \tau_i$. Then, the Taylor formula yields
\ba
\label{e:f}
\E_{x,1}f(Z^{\e,\delta}_{T_p^\e})& -f(x)
=\sum_{i=1}^{p} \E_{x,1}\Big[f(Z^{\e,\delta}_{T_{i}^\e})- f(Z^{\e,\delta}_{T_{i-1}^\e})\Big]\\
&= \sum_{i=1}^{p}\E_{x,1}
\Big[f \Big(Z^{\e,\delta}_{T_{i-1}^\e} + \frac{1}{\delta}V_{i}(Z^{\e,\delta}_{T_{i-1}^\e} )\Delta_i T^\e\Big)
-f (Z^{\e,\delta}_{T_{i-1}^\e} )\Big]\\
&= \frac{1}{\delta}\sum_{i=1}^{p}
\E_{x,1}
\langle \nabla f(Z^{\e,\delta}_{T_{i-1}^\e} ) , V_{i} (Z^{\e,\delta}_{T_{i-1}^\e} )\Delta_i T^\e\rangle\\
&+ \frac{1}{\delta^2} \sum_{i=1}^{p}
\E_{x,1} \int_0^1  \Big\langle \nabla^2 f\Big(Z^{\e,\delta}_{T_{i-1}^\e}
+  \frac{\theta}{\delta}V_{i} (Z^{\e,\delta}_{T_{i-1}^\e} )\Delta_i T^\e\Big) V_i (Z^{\e,\delta}_{T_{i-1}^\e} )\Delta_i T^\e,
V_{i} (Z^{\e,\delta}_{T_{i-1}^\e} )\Delta_i T^\e\Big\rangle(1-\theta)\,\di \theta.
\ea
a) We expand and estimate the first term in the last sum in \eqref{e:f}.
To exploit the centering condition in the first-order Taylor contribution, set
\ba
\label{e:g}
g_i(x):=\langle \nabla f(x),  V_i (x)\rangle.
\ea
Then $\mathbf{A}_{V,\mu}$ implies
\ba
\sum_{i=1}^{p } \mu_i g_i(x) = \Big\langle \nabla f(x), \sum_{i=1}^{p }\mu_i V_i(x)\Big\rangle \equiv 0,\quad
x\in\bR^d.
\ea
Consequently
we get
\ba
\frac{1}{\delta}\sum_{i=1}^{p} \E_{x,1} \Big[ g_i (Z^{\e,\delta}_{T_{i-1}^\e} ) \Delta_i T^\e\Big]
&=  \frac{\e}{\delta} \sum_{i=1}^{p} \mu_i \E_{x,1}\Big[ g_i(Z^{\e,\delta}_{T_{i-1}^\e} ) - g_i(x) \Big]\\
&=  \frac{\e}{\delta} \sum_{i=2}^{p}  \mu_i \E_{x,1}\Big[ g_i(Z^{\e,\delta}_{T_{i-1}^\e} ) - g_i(x) \Big]\\
&=   \frac{\e}{\delta}\sum_{i=2}^{p}\mu_i\sum_{j=1}^{i-1}   \E_{x,1}
\Big[ g_i(Z^{\e,\delta}_{T_{j}^\e} ) - g_i(Z^{\e,\delta}_{T^\e_{j-1}}) \Big].
\ea
For each $j<i$, we have
\ba
 \frac{\e}{\delta}\E_{x,1}[g_i(Z^{\e,\delta}_{T_{j}^\e} ) &- g_i(Z^{\e,\delta}_{T^\e_{j-1}})]
= \frac{\e}{\delta^2}\E_{x,1}\Big[ \Delta_j T^\e
\int_0^1 \langle \nabla
g_i(Z^{\e,\delta}_{T_{j-1}^\e}
+\frac{\theta \Delta_j T^\e}{\delta }
V_j(Z^{\e,\delta}_{T_{j-1}^\e} )), V_j(Z^{\e,\delta}_{T_{j-1}^\e} )\rangle\,\di \theta \Big]\\
&=  \frac{\e^2}{\delta^2} \mu_j \langle \nabla g_i(x), V_j(x) \rangle\\
&+
\frac{\e}{\delta^2}
\E_{x,1}\Big[\Delta_j T^\e
\int_0^1 \Big(\langle \nabla
g_i(Z^{\e,\delta}_{T_{j-1}^\e}
+ \frac{\theta \Delta_j T^\e}{\delta }V_j(Z^{\e,\delta}_{T_{j-1}^\e} )), V_j(Z^{\e,\delta}_{T_{j-1}^\e} ) \rangle
- \langle \nabla g_i(x), V_j(x)\rangle\Big)\,\di \theta \Big]\\
&= \frac{\e^2}{\delta^2} \mu_j \langle \nabla g_i(x), V_j(x) \rangle + \E_{x,1}R^{\e,\delta}_{ij}.
\ea
Overall,
\ba
\frac{1}{\delta}\sum_{i=1}^{p}
\E_{x,1}
\langle \nabla f(Z^{\e,\delta}_{T_{i-1}^\e} ) , V_{i} (Z^{\e,\delta}_{T_{i-1}^\e} )\Delta_i T^\e\rangle
= \frac{\e^2}{\delta^2}\sum_{i=2}^{p}\sum_{j=1}^{i-1}    \mu_i\mu_j \langle \nabla g_i(x), V_j(x) \rangle +
 \sum_{i=2}^{p}\sum_{j=1}^{i-1}    \mu_i \E_{x,1}R^{\e,\delta}_{ij}.
\ea
Notice that only indices $j<i$ appear in our analysis:
the value of $V_i$ at the beginning of the $i$-th run is affected by the displacements
generated by the preceding vector fields $V_1,\dots,V_{i-1}$.
This is the point at which the prescribed cyclic ordering enters the limiting drift.

To estimate the remainder, we note that
for all $x,y,z\in\bR^d$
\ba
|\langle\nabla g_i(y+ z), V_j(y) \rangle - \langle \nabla g_i(x), V_j(x)\rangle|
&\leq
|\langle\nabla g_i(y+ z) - \nabla g_i(y), V_j(y) \rangle |\\
&+|\langle\nabla g_i(y), V_j(y) \rangle - \langle \nabla g_i(x), V_j(x) \rangle|\\
&\leq C(|z|  + |x-y| ),
\ea
where the constant $C$ depends only on $\|\nabla^k f\|$, $k=1,2$, and $\|\nabla^k V\|$, $k=0,1$.
Thus, taking into account the Kac scaling $\e/\delta^2\to \varkappa^2$ we get
\ba
\E_{x,1} | R^{\e,\delta}_{ij}|
&\leq_C
\frac{\e}{\delta^2}
\E_{x,1}\Big[\Delta_j T^\e
\int_0^1 \Big(\frac{\theta \Delta_j T^\e}{\delta }|V_j(Z^{\e,\delta}_{T_{j-1}^\e})|
+ |Z^{\e,\delta}_{T_{j-1}^\e}-x|\Big)\,\di \theta \Big]\\
& \leq_C
\frac{\e}{\delta^3 }\E (\Delta_j T^\e)^2
+ \frac{\e}{\delta^2 }\E\Delta_j T^\e \E_{x,1}  |Z^{\e,\delta}_{T_{j-1}^\e}-x| \\
&\leq_C \frac{\e^3 }{\delta^3}
+ \frac{\e^2}{\delta^2 }\sum_{k=1}^{j-1}\E_{x,1}|Z^{\e,\delta}_{T_{k}^\e}-Z^{\e,\delta}_{T_{k-1}^\e}|\\
&\leq_C  \frac{\e^3 }{\delta^3}
+ \frac{\e^3}{\delta^4}
\\
&\leq_C  \e .
\ea
b) We estimate the second term in the last sum in \eqref{e:f}. For each $i=1,\dots,p$ and $x,y,z\in\bR^d$ we have
\ba
\Big|\langle \nabla^2 f(y + z) V_i (y), V_i (y)\rangle  - \langle \nabla^2 f(x) V_i(x), V_i(x)\rangle\Big|
&\leq
\Big|\langle \nabla^2 f(y + z) V_i(y), V_i(y)\rangle - \langle \nabla^2 f(y) V_i(y), V_i (y)\rangle\Big|\\
&+\Big|\langle \nabla^2 f(y) V_i (y), V_i (y)\rangle - \langle \nabla^2 f(x) V_i (x), V_i (x)\rangle\Big|.
\ea
Note that $z\mapsto \langle \nabla^2 f(y + z) V_i(y), V_i(y)\rangle$ is bounded and smooth.
We use here the weaker H\"older-type bound with exponent $\rho\in (0,1)$ in order to estimate
the remainder using only the assumed $(2+\rho)$-moments of the run times.
We have
\ba
\Big|\langle \nabla^2 f(y + z) V_i(y), V_i(y)\rangle - \langle \nabla^2 f(y) V_i(y), V_i (y)\rangle\Big|
\leq_C |z|\wedge 1
\leq_C  |z|^\rho\wedge 1
\leq_C  |z|^{\rho},
\ea
where the constant in this estimate depends only on $\|\nabla^k f\|$, $k=2,3$, and $\|V\|$.
Analogously, since $y\mapsto \langle \nabla^2 f(y) V_i (y), V_i (y)\rangle$ is bounded and smooth, we have
\ba
\Big|\langle \nabla^2 f(y) V_i (y), V_i (y)\rangle - \langle \nabla^2 f(x) V_i (x), V_i (x)\rangle\Big|
\leq_C|x-y|^{\rho},
\ea
where the constant depends on $\|\nabla^k f\|$, $k=2,3$, and $\|\nabla^k V\|$, $k=0,1$.
Therefore,
\ba
& \frac{1}{\delta^2}\sum_{i=1}^p
\E_{x,1} \int_0^1  \Big\langle \nabla^2 f\Big(Z^{\e,\delta}_{T_{i-1}^\e}
+  \frac{\theta}{\delta}V_i (Z^{\e,\delta}_{T_{i-1}^\e} )\Delta_i T^\e\Big) V_i (Z^{\e,\delta}_{T_{i-1}^\e} )\Delta_i T^\e,
V_i (Z^{\e,\delta}_{T_{i-1}^\e} )\Delta_i T^\e\Big\rangle(1-\theta)\,\di \theta\\
&=  \frac{\e^2}{2\delta^2}\sum_{i=1}^p \lambda_i^2
\langle \nabla^2 f(x)V_i (x), V_i (x)\rangle
+\sum_{i=1}^p \E_{x,1} R^{\e,\delta}_i,
\ea
where for each $i=1,\dots,p$
\ba
\E_{x,1} |R^{\e,\delta}_i|
&\leq_C
\frac{1}{\delta^2}
\E_{x,1} \int_0^1 (\Delta_i T^\e)^2\Big(
\Big| \frac{\theta}{\delta}V_i (Z^{\e,\delta}_{T_{i-1}^\e})\Delta_i T^\e\Big|^\rho
+ |Z^{\e,\delta}_{T_{i-1}^\e} - x |^\rho
\Big)(1-\theta)\,\di \theta\\
&\leq_C
\frac{1}{\delta^{2+\rho}}
\E_{x,1} (\Delta_i T^\e)^{2+\rho}
+
\frac{1}{\delta^{2}}
\E_{x,1}\Big[ (\Delta_i T^\e)^{2}
\sum_{j=1}^{i-1}|Z^{\e,\delta}_{T_{j}^\e} - Z^{\e,\delta}_{T_{j-1}^\e}|^\rho\Big]\\
&\leq_C
\frac{1}{\delta^{2+\rho}}
\E_{x,1} (\Delta_i T^\e)^{2+\rho}
+
\frac{1}{\delta^{2+\rho}}\sum_{j=1}^{i-1}
\E_{x,1}\Big[ (\Delta_i T^\e)^{2}(\Delta_j T^\e)^\rho\Big]\\
&\leq_C \frac{\e^{2+\rho}}{\delta^{2+\rho}}
\leq_C  \e^{1+\frac{\rho}{2}}.
\ea
2. Combining the results from a) and b) we get
\ba
\label{e:opA1}
\frac{\E_{x,1}f(Z^{\e,\delta}_{T_p^\e}) -f(x)}{\E_1 T^\e_p}
&=\frac{1}{\mu\e}\frac{\e^2}{\delta^2}
\Big(
\sum_{i=2}^{p}\sum_{j=1}^{i-1}
\mu_i\mu_j \langle \nabla \langle \nabla f(x),  V_i (x)\rangle, V_j(x) \rangle
+\frac{1}{2}  \sum_{i=1}^{p}  \lambda_i^2
\langle \nabla^2 f(x)V_i (x), V_i (x)\rangle
+\E_{x,1} R^{\e,\delta}\Big) \\
&=\frac{\varkappa^2}{\mu}
\Big(
\sum_{j<i} \mu_i\mu_j \langle \nabla \langle \nabla f(x),  V_i (x)\rangle, V_j(x) \rangle
+\frac{1}{2}  \sum_i  \lambda_i^2 \langle \nabla^2 f(x)V_i (x), V_i (x)\rangle \Big)
+ r^{\e,\delta}_{x,1},
\ea
where
\ba
\lim_{\e,\delta\to 0}|r^{\e,\delta}_{x,1}|=0.
\ea
Thus, before using the centering assumption $\mathbf{A}_{V,\mu}$ a second time,
the limiting second-order contribution contains both the individual quadratic
terms associated with each vector field and cross terms generated by successive vector
fields within the cycle.

Recall the identity
\ba
\langle \nabla \langle \nabla f,  V_i\rangle, V_j \rangle
= \langle \nabla^2 f  V_i , V_j \rangle + \langle \nabla f,  D V_i[V_j] \rangle,
\ea
where $D V_i[V_j]=(DV_i)V_j$ is the directional derivative of $V_i$ in direction $V_j$ defined in \eqref{e:DVV}.
We can rewrite
\ba
\label{e:H1}
\frac{\E_{x,1}f(Z^{\e,\delta}_{T_p^\e}) -f(x)}{\E_1 T^\e_p}
&= \frac{\varkappa^2}{\mu} \sum_{j<i} \mu_i\mu_j \langle \nabla^2 f(x)  V_i(x) , V_j(x) \rangle
+ \frac{\varkappa^2}{\mu} \sum_{j<i} \mu_i\mu_j  \langle \nabla f,  D V_i[V_j](x) \rangle\\
&+\frac{\varkappa^2}{2\mu} \sum_{i}  \lambda_i^2 \langle \nabla^2 f(x)V_i(x), V_i(x)\rangle
+r^{\e,\delta}_{x,1}
\ea
The Hessian cross terms can be simplified once more using the centering assumption $\mathbf{A}_{V,\mu}$:
\ba
\label{e:H2}
0=\Big\langle \nabla^2 f \sum_{i=1}^{p} \mu_i V_i,  \sum_{j=1}^{p} \mu_j V_j\Big\rangle
=\sum_{i} \mu_i^2 \langle \nabla^2 f V_i, V_i\rangle
+ 2 \sum_{j<i} \mu_i\mu_j \langle \nabla^2 f V_i, V_j\rangle,
\ea
Substituting \eqref{e:H2} into \eqref{e:H1} replaces the second moments $\lambda_i^2=\sigma_i^2+\mu_i^2$ by the variances $\sigma_i^2$,
leaving the order-dependent first-order term unchanged, so that
\ba
\label{e:AA}
\frac{\E_{x,1}f(Z^{\e,\delta}_{T_p^\e}) -f(x)}{\E_1 T^\e_p}
&=
\frac{\varkappa^2}{2\mu} \sum_i  \sigma_{i}^2
\langle \nabla^2 f(x)V_i(x), V_i(x)\rangle
+\frac{\varkappa^2}{\mu} \sum_{j<i}
\mu_i\mu_j  \langle \nabla f(x),  D V_i[V_j](x) \rangle
+r^{\e,\delta}_{x,1}\\
&= Af(x) + +r^{\e,\delta}_{x,1},
\ea
with the operator $A$ defined in \eqref{e:opA}.

\smallskip
\noindent
3.
It remains to show that the limiting operator is independent of the initial direction $\iota$.
The preceding remainder estimates are uniform in $x$ and $\iota$,
so it suffices to verify that the limiting differential operator is invariant under cyclic shifts of the indices.

Let
\ba
Af=A_1 f:=\frac{\varkappa^2}{2\mu} \sum_i  \sigma_{i}^2
\langle \nabla^2 fV_i , V_i \rangle
+\frac{\varkappa^2}{\mu} \sum_{j<i}
\mu_i\mu_j  \langle \nabla f ,  D V_i[V_j]  \rangle
\ea
be given in \eqref{e:AA}. Let the initial direction be $\iota=2$, and let
 \ba
A_2 f & :=\frac{\varkappa^2}{2\mu} \sum_i  \sigma_{1+i\mathrm{mod}\,p}^2
\langle \nabla^2 f(x)V_{1+i\mathrm{mod}\,p}(x), V_{1+i\mathrm{mod}\,p}(x)\rangle\\
&+\frac{\varkappa^2}{\mu}
\sum_{j<i} \mu_{1+i\mathrm{mod}\,p}\mu_{1+j\mathrm{mod}\,p}  \langle \nabla f,
D V_{1+i\mathrm{mod}\,p}[V_{1+j\mathrm{mod}\,p}] \rangle
\ea
be the corresponding limiting operator with all indices cyclically shifted by 1.
It is clear that the second order parts of $A_1$ and $A_2$ coincide. Therefore,
\ba
A_2f - A_1 f &=\sum_{j=2}^{p} \mu_1\mu_j \Big(\langle \nabla f,  D V_1[V_j] \rangle - \langle \nabla f,  D V_j[V_1] \rangle\Big)\\
&=\mu_1 \Big\langle\nabla f, DV_1\Big[ \sum_{j=2}^{p} \mu_jV_j   \Big]  -  \sum_{j=2}^{p} \mu_j DV_j[V_1]    \Big\rangle.
\ea
Since by assumption $\mathbf{A}_{V,\mu}$
\ba
 \sum_{j=2}^{p} \mu_j V_j=-\mu_1 V_1,
\ea
we get
\ba
  DV_1\Big[ \sum_{j=2}^{p} \mu_jV_j   \Big] =-\mu_1 DV_1[V_1]
\ea
and
\ba
\sum_{j=2}^{p} \mu_j DV_j[V_1] = \Big(D\sum_{j=2}^{p} \mu_jV_j\Big)[V_1]
=-\mu_1 DV_1[V_1].
\ea
Therefore,
\ba
A_2f - A_1f &=\mu_1 \Big\langle\nabla f, -DV_1 [ V_1 ] +   DV_1 [ V_1 ]    \Big\rangle\equiv  0.
\ea
Repeating the same argument for successive cyclic shifts shows that the limiting operator
is independent of the initial direction $\iota$. Notice that this invariance concerns cyclic shifts only.
Example \ref{exa:UV} illustrates that changing the cyclic ordering itself may change the limiting drift.
\end{proof}

\section{Martingale problem and proof of Theorem \ref{t:A}\label{s:Aconv}}

We now combine the approximation of the pseudo-characteristic operator obtained in section \ref{s:operator} with the relative
compactness established in section \ref{s:Awrc}. We first verify an approximate martingale problem for $\{Z^{\e,\delta}\}$,
then identify every weak limit with the diffusion generated by $A$.

\begin{lem}
\label{l:estA}
For any $f\in C^3_b(\bR^d,\bR)$ and
any
$\eta\in(0,1]$
\ba
\E_{x,\iota}\Big| f(Z^{\e,\delta}_{T_p^\e}) - f(x) - \int_0^{T^\e_p} Af(Z^{\e,\delta}_u)\,\di u\Big|
\leq \eta \E_{\iota} T^\e_p
\ea
uniformly in $x\in\bR^d$ and $\iota\in\{1,\dots,p\}$ for sufficiently small $\e,\delta$.
Consequently,
\ba
\Big|\E\Big[ f(Z^{\e,\delta}_{T_{k+p}^\e}) - f(Z^{\e,\delta}_{T_{k}^\e})
- \int_{T_{k}^\e}^{T_{k+p}^\e} Af(Z^{\e,\delta}_u) \,\di u\     \Big|\rF^{\e,\delta}_{T_{k}^\e} \Big]\Big|
\leq \eta  \E\Big[ T_{k+p}^\e-T_{k}^\e \Big|\rF^{\e,\delta}_{T_{k}^\e} \Big]
\ea
uniformly in  $k\in\mathbb N_0$ for sufficiently small $\e,\delta$.
\end{lem}
\begin{proof}
By Lemma \ref{l:A}, for any $\eta\in(0,1]$ and $\e,\delta\in(0,1]$ small enough
\ba
\E_{x,\iota} f(Z^{\e,\delta}_{T^\e_{p}}) - f(x)
&\leq \Big(A f(x) +\frac{\eta}{2}\Big)\E_{\iota}T_p^\e\\
&= \E_{x,\iota}\int_0^{T_p^\e} \Big(A f(Z^{\e,\delta}_u) + \frac{\eta}{2}\Big)\,\di u
-\E_{x,\iota}\int_0^{T_p^\e} \Big(A f(Z^{\e,\delta}_u) -Af(x)\Big)\,\di u.
\ea
Under assumption $\mathbf{A}_V$, $Af\in C^1_b(\bR^d,\bR)$ for  $f\in C^3_b(\bR^d,\bR)$, so that
for $u\in[0,T_p^\e]$
\ba
|A f(Z^{\e,\delta}_u) -Af(x)| & \leq \|\nabla (Af)\| \max_{1\leq i\leq p}|Z^{\e,\delta}_{T^\e_i} - x|\\
&\leq \|\nabla (Af)\| \|V\|   \frac{1}{\delta }T_p^\e.
\ea
Hence, denoting $\widehat \tau=\tau_1+\cdots+\tau_p$ we estimate
\ba
\E_{x,\iota}\int_0^{T_p^\e} \Big|A f(Z^{\e,\delta}_u) -Af(x)\Big|\,\di u
&\leq \|\nabla (Af)\| \|V\|
\frac{1}{\delta}\E |T_p^\e|^2\\
&=\|\nabla (Af)\| \|V\| \E\widehat\tau^2 \cdot
\frac{\e^2}{\delta}\\
&=\E T_p^\e \cdot \frac{\|\nabla (Af)\| \|V\| \E\widehat\tau^2}{\mu} \frac{\e}{\delta}
\leq\E_{x,\iota}\int_0^{T_p^\e} \frac{\eta}{2}\,\di u
\ea
for $\e,\delta$ small such that $   \frac{\|\nabla (Af)\| \|V\| \E\widehat\tau^2}{\mu} \frac{\e}{\delta}   \leq\frac{\eta}{2}$.
The estimate from below follows analogously.
\end{proof}

\begin{lem}
\label{l:mp}
Let $f\in C^3_b(\bR^d,\bR)$ and $0\leq s<t<\infty$. Then
\ba
\label{e:mp}
\E\Big[ f(Z^{\e,\delta}_{t}) - f( Z^{\e,\delta}_s  ) - \int_s^t Af(Z^{\e,\delta}_u)\,\di u\Big|\rF^{\e,\delta}_s\Big]
\to 0
\ea
in probability as $\e,\delta\to 0$.
\end{lem}
\begin{proof}
Let
\ba
\widehat N^\e_t&:=\min\{k\in\mathbb N_0 \colon T^\e_{kp}> t\}\\
\ea
be the first passage process associated with the sequence $\widehat T^\e=\{T^\e_{kp}\}_{k\in\mathbb N_0}$.
The latter is a renewal process with iid increments. We replace the deterministic times $s$ and $t$
by the first complete-cycle renewal epochs following them and decompose the expression in
\eqref{e:mp}
as follows:
\ba
\label{e:fdd}
\E\Big[f(Z^{\e,\delta}_t) - f(Z^{\e,\delta}_s) & - \int_s^t Af(Z^{\e,\delta}_u)\,\di u\Big|\rF_s^{\e,\delta}\Big]\\
&=\E\Big[f(Z^{\e,\delta}_{T^\e_{p \widehat N^\e_t}}) - f(Z^{\e,\delta}_{T^\e_{p \widehat N^\e_s}})
- \int^{T^\e_{p \widehat N^\e_t}}_{T^\e_{p \widehat N^\e_s}}  Af(Z^{\e,\delta}_u)\,\di u\Big|\rF_s^{\e,\delta}\Big]\\
&+ \E\Big[f(Z^{\e,\delta}_{T^\e_{p \widehat N^\e_s}}) - f(Z^{\e,\delta}_{s})
- \int_{s}^{T^\e_{p \widehat N^\e_s}} Af(Z^{\e,\delta}_u)\,\di u\Big|\rF_s^{\e,\delta}\Big]\\
&-\E\Big[ f(Z^{\e,\delta}_{T^\e_{p \widehat N^\e_t}}) - f(Z^{\e,\delta}_t)
-\int_{t}^{T^\e_{p \widehat N^\e_t}} Af(Z^{\e,\delta}_u)\,\di u\Big|\rF_s^{\e,\delta}\Big].
\ea
We estimate the first summand in \eqref{e:fdd} with the help of Lemma \ref{l:estA}. Let $\eta\in(0,1]$ be arbitrary. Then,
taking into account that $\{T^\e_{kp}\in[s,t)\} \in \rF^{\e,\delta}_{T^\e_{kp}}$ we get
\ba
\E\Big[f(Z^{\e,\delta}_{T^\e_{p \widehat N^\e_t}}) & - f(Z^{\e,\delta}_{T^\e_{p \widehat N^\e_s}})
- \int^{T^\e_{p \widehat N^\e_t}}_{T^\e_{p \widehat N^\e_s}}  Af(Z^{\e,\delta}_u)\,\di u\Big|\rF_s^{\e,\delta}\Big]\\
&=\E\Big[\sum_{k=0}^\infty\bI_{T^\e_{kp}\in[s,t)} \Big(f(Z^{\e,\delta}_{T^\e_{(k+1)p}})
- f(Z^{\e,\delta}_{T^\e_{kp}})
- \int^{T^\e_{(k+1)p}}_{T^\e_{kp}} Af(Z^{\e,\delta}_u)\,\di u\Big)\Big|\rF_s^{\e,\delta}\Big]\\
&=\sum_{k=0}^\infty\E \Big[\bI_{T^\e_{kp}\in[s,t)} \E\Big[f(Z^{\e,\delta}_{T^\e_{(k+1)p}}) - f(Z^{\e,\delta}_{T^\e_{kp}})
- \int_{T^\e_{kp}}^{T^\e_{(k+1)p}} Af(Z^{\e,\delta}_u)\,\di u\Big|\rF^{\e,\delta}_{T^\e_{kp}} \Big]  \Big|\rF_s^{\e,\delta}\Big]\\
&\leq\sum_{k=0}^\infty
\E \Big[\bI_{T^\e_{kp}\in[s,t)}
\E\Big[ (T^\e_{(k+1)p}-T^\e_{kp})\eta \Big|\rF^{\e,\delta}_{T^\e_{kp}} \Big]  \Big|\rF_s^{\e,\delta}\Big]\\
&\leq \eta (t-s)
+ \eta \E\Big[T^\e_{p \widehat N^\e_t}-T^\e_{p \widehat N^\e_t-p} \Big|\rF_s^{\e,\delta}\Big].
\ea
The last term is estimated from above by means of
the estimate from \cite[Section 2]{lorden1970excess}. We have
\ba
\E_{x,\iota} (T^\e_{p \widehat N^\e_t}-T^\e_{p \widehat N^\e_t-p})\leq 2\frac{\E |T^\e_1|^2}{\E T^\e_1}
=2\frac{\E\widehat \tau^2}{\E \widehat \tau}\e.
\ea
The estimate from below is obtained analogously.

For the second summand in \eqref{e:fdd} we get
\ba
\E_{x,\iota}\Big|f(Z^{\e,\delta}_{T^\e_{p \widehat N^\e_s}}) - f(Z^{\e,\delta}_{s})
&- \int_{s}^{T^\e_{p \widehat N^\e_s}} Af(Z^{\e,\delta}_u)\,\di u\Big|\\
&\leq
\E_{x,\iota}\Big|f(Z^{\e,\delta}_{T^\e_{p \widehat N^\e_s}}) - f(Z^{\e,\delta}_s )\Big|
+\E_{x,\iota} \Big|\int_s^{T^\e_{p \widehat N^\e_s}} Af(Z^{\e,\delta}_u)\,\di u\Big|\\
&\leq
\E_{x,\iota}\max_{i=1,\dots,p}\Big|f(Z^{\e,\delta}_{T^\e_{p \widehat N^\e_s}}) - f(Z^{\e,\delta}_{T^\e_{p \widehat N^\e_s-i}} )\Big|
+\|Af\|\E_\iota (T^\e_{p \widehat N^\e_s}-s)\\
&\leq \|\nabla f\|\|V\|\frac{1}{\delta}\E_{x,\iota} (T^\e_{p \widehat N^\e_s}-T^\e_{p \widehat N^\e_s-p})
+\|Af\|\E_\iota (T^\e_{p \widehat N^\e_s}-s).
\ea
Again, by \cite[Section 2]{lorden1970excess} we have
$
\E_{x,\iota} (T^\e_{p \widehat N^\e_s}-T^\e_{p \widehat N^\e_s-p})\leq 2\frac{\E\widehat \tau^2}{\E \widehat \tau}\e,
$
whereas for the overshoot we have $\E_\iota (T^\e_{p \widehat N^\e_s}-s)\leq  \frac{\E\widehat \tau^2}{\E \widehat \tau} \e$.
Combining these estimates yields
\ba
\E_{x,\iota}\Big|f(Z^{\e,\delta}_{T^\e_{p \widehat N^\e_s}}) - f(Z^{\e,\delta}_{s})
- \int_{s}^{T^\e_{p \widehat N^\e_s}} Af(Z^{\e,\delta}_u)\,\di u\Big|\leq_C  \frac{\e}{\delta} +\e\leq_C\sqrt\e.
\ea
The same estimate holds for the third summand  in \eqref{e:fdd}, and \eqref{e:mp} follows.
\end{proof}

Finally we prove Theorem \ref{t:A}. By section \ref{s:Awrc}, the family $\{Z^{\e,\delta}\}$ is
weakly relatively compact in $D([0,\infty),\bR^d)$. Let $Z$ be any weak limit point.
Passing to the limit in \eqref{e:mp}, after multiplication by bounded continuous functions
of finitely many values of the process up to time $s$, shows that every weak limit $Z$
solves the martingale problem for $A$.
Since the coefficients of $A$ are bounded and sufficiently regular, this martingale problem is well posed.
Hence $Z$ has the law of the diffusion $X$ from Theorem 2.3. Consequently,
$Z^{\e,\delta} \Rightarrow X$ in $D([0,\infty),\bR^d)$.
By Lemma \ref{l:ucp},
$\sup_{t\in [0,T]}|Z^{\e,\delta}_t-X^{\e,\delta}_t|\to 0$ in probability for every $T\in[0,\infty)$.
Hence, the converging-together theorem yields $X^{\e,\delta} \Rightarrow X$ in $D([0,\infty),\bR^d)$.
Since the limit $X$ has continuous paths, this convergence is equivalent to weak convergence
with respect to the topology of uniform convergence on compact time intervals.

The derivation of the SDE \eqref{e:SDEX} is straightforward. Recalling the It\^o--Stratonovich correction, we rewrite the SDE
\eqref{e:SDEX} in the Stratonovich form as
\ba
\label{e:Strat1}
\di X_t= \frac{\varkappa}{\sqrt\mu}  \sum_{i} \sigma_{i} V_i(X_t)\circ \di W^i_t -
\frac{\varkappa^2}{2\mu} \sum_{i} \sigma^2_i  D V_i [ V_i](X_t)\,\di t+
\frac{\varkappa^2}{\mu} \sum_{j< i} \mu_i\mu_j  D V_i [ V_j](X_t)\,\di t.
\ea
Furthermore, we use assumption $\mathbf{A}_{V,\mu}$ to obtain
\ba
0&\equiv D\Big(\sum_i\mu_iV_i\Big)\Big[\sum_j\mu_j V_j \Big]=\sum_{i,j} \mu_i\mu_j  DV_i[V_j]\\
&=\sum_{i} \mu_i^2 DV_i[V_i] + \sum_{j<i} \mu_i\mu_j \Big(DV_i[V_j]+DV_j[V_i]\Big).
\ea
Since $DV_j[V_i]= [V_i,V_j] + DV_i[V_j]$, we get
\ba
0= \sum_{i}\mu_i^2 DV_i[V_i] + 2\sum_{j<i} \mu_i\mu_j  DV_i[V_j] + \sum_{j<i} \mu_i\mu_j  [V_i,V_j]
\ea
or, equivalently,
\ba
\label{e:DL}
\sum_{j<i} \mu_i\mu_j  DV_i[V_j] = -\frac12 \sum_{i}\mu_i^2 DV_i[V_i] -\frac12 \sum_{j<i} \mu_i\mu_j  [V_i,V_j].
\ea
Substituting \eqref{e:DL} into \eqref{e:Strat1} yields
\ba
\di X_t
&= \frac{\varkappa}{\sqrt\mu}  \sum_{i} \sigma_{i} V_i(X_t)\circ \di W^i_t -
\frac{\varkappa^2}{2\mu} \sum_{i} (\sigma^2_i+\mu_i^2)  D V_i [ V_i](X_t)\,\di t
-\frac{\varkappa^2}{2\mu}\sum_{j<i} \mu_i\mu_j  [V_i,V_j](X_t)\,\di t.
\ea
Finally, using $\sigma^2_i + \mu_i^2=\lambda^2_i$ we obtain \eqref{e:XStrat}.

\section{Diffusion approximation for the gliding dynamics\label{s:gliding}}

The proof of Theorem \ref{t:B} follows the same scheme as that of Theorem \ref{t:A}.

Let $\Psi_i^\delta(y,t)$, $i=1,\dots,p$, be the unique solution of the ODE
\ba
\label{e:Psi}
\dot \Psi_i^\delta=\frac{1}{\delta}V_i(\Psi^\delta_i),\quad \Psi_i^\delta(y,0)=y\in\bR^d,\quad t\in[0,\infty).
\ea
Then, on each interval $t\in[T^\e_k,T^\e_{k+1}]$, $k\in\mathbb N_0$, we have
\ba
Y^{\e,\delta}_t=\Psi_{n_k}^\delta(Y^{\e,\delta}_{T^\e_k},t-T_k^\e).
\ea
Since the vector fields $V_i$ are bounded,
\ba
|Y^{\e,\delta}_t - Y^{\e,\delta}_{T^\e_k}|=
\frac{1}{\delta} \Big|\int_{0}^{t -T^\e_k} V_{n_k}(\Psi^\delta_{n_k}(Y^{\e,\delta}_{T^\e_k} ,s  ))\,\di s  \Big|
\leq   \frac{\|V_{n_k}\|}{\delta}(t-T^\e_k),\quad t\in[T^\e_k,T^\e_{k+1}],
\ea
which is the same within-run estimate as for $X^{\e,\delta}$.
Together with the centering condition and the Taylor estimates for the flows,
this allows the compact-containment and Aldous arguments of section~\ref{s:Awrc} to be repeated for the embedded step process
$\widehat Z^{\e,\delta}_t:=Y^{\e,\delta}_{T_k^\e}$, $t\in [T_k^\e,T^\e_{k+1})$,
$k\in\mathbb N_0$, with only notational modifications.
Once the analogue of the pseudo-characteristic operator approximation
is established below, the martingale-problem argument of section \ref{s:Aconv} applies with $B$ in place of $A$.

The only essential modification concerns the expansion of the $p$-step pseudo-characteristic operator,
since in the gliding model the particle follows the flow of the current vector field during each run.
In contrast to the flight dynamics, the vector field is not frozen during a run, and its variation
along its own flow generates an additional second-order contribution.

For $f\in C^3_b(\bR^d,\bR)$ define
\ba
B_\iota^{\e,\delta }f(y)&=\frac{\E_{y,\iota} f(Y^{\e,\delta}_{T^\e_{p}}) - f(y)}{\E_\iota T_p^\e }.
\ea

\begin{lem}
\label{l:B}
For every $f\in C^3_b(\bR^d,\bR)$ we have
\ba
\label{e:Biota}
\lim_{\e,\delta \downarrow 0} \sup_{y\in\mathbb R^d,\iota\in\{1,\dots,p\}}
|B_\iota^{\e,\delta }f(y) -B f(y)|=0,
\ea
where the operator $B$ is defined in \eqref{e:opB}.
\end{lem}

\begin{proof}

We follow the proof of Lemma \ref{l:A} and indicate only the terms that differ in the gliding case.
For definiteness, we first take $\iota=1$.

Let $f\in C^3_b(\bR^d,\bR)$ and
$y\in\bR^d$. Using the flow equation \eqref{e:Psi}
we decompose the increment into the contribution obtained by freezing the vector
field at the beginning of the run and a correction accounting for its variation along the flow.
As in \eqref{e:g}, let $g_i(y):=\langle \nabla f(y),V_i(y)\rangle$, $i=1,\dots, p$
and write
\ba
\E_{y,1}f(Y^{\e,\delta}_{T_p^\e}) -f(y)
&=\sum_{i=1}^{p} \E_{y,1}\Big[f(Y^{\e,\delta}_{T_{i}^\e})- f(Y^{\e,\delta}_{T_{i-1}^\e})\Big]\\
&= \sum_{i=1}^{p}\E_{y,1}
\Big[f (\Psi_{i}^\delta (Y^{\e,\delta}_{T_{i-1}^\e},\Delta T^\e_i))
-f (Y^{\e,\delta}_{T_{i-1}^\e} )\Big]\\
&= \frac{1}{\delta}\sum_{i=1}^{p}
\E_{y,1}\int_0^1 g_i(\Psi^\delta_i (Y^{\e,\delta}_{T_{i-1}^\e},\theta \Delta T_{i}^\e)) \Delta T_{i}^\e \,\di \theta \\
&= \frac{1}{\delta}\sum_{i=1}^{p}\E_{y,1}\Big[ g_i(Y^{\e,\delta}_{T_{i-1}^\e})\Delta T_{i}^\e\Big] \\
&+ \frac{1}{\delta}\sum_{i=1}^{p}
\E_{y,1}\int_0^1 \Big( g_i(\Psi^\delta_i (Y^{\e,\delta}_{T_{i-1}^\e},\theta\Delta T_{i}^\e))
- g_i(Y^{\e,\delta}_{T_{i-1}^\e}) \Big) \Delta T_{i}^\e \,\di \theta.
\ea
The first term has exactly the same form as the corresponding first-order contribution in the flight model.
The second term is specific to the gliding dynamics.
\ba
\frac{1}{\delta}
\E_{y,1}\int_0^1 &\Big( g_i(\Psi^\delta_i (Y^{\e,\delta}_{T_{i-1}^\e},\theta\Delta T_{i}^\e))
-  g_i(Y^{\e,\delta}_{T_{i-1}^\e}) \Big) \Delta T_{i}^\e \,\di \theta \\
&=
\frac{1}{\delta^2}
\E_{y,1}\int_0^1\int_0^1 \langle \nabla g_i(\Psi^\delta_i (Y^{\e,\delta}_{T_{i-1}^\e},\theta'\theta \Delta T_{i}^\e)),
V_i(  \Psi^\delta_i (Y^{\e,\delta}_{T_{i-1}^\e},\theta'\theta \Delta T_{i}^\e )) \rangle \theta(\Delta T_{i}^\e)^2
\,\di \theta' \,\di \theta.
\ea
Expanding first along the $i$-th flow and then replacing $Y^{\e,\delta}_{T_{i-1}^\e}$
by $y$, with the resulting error estimated as in Lemma~\ref{l:A}, we obtain
\ba
\cdots &=
\frac{\e^2}{2\delta^2}\lambda_i^2  \langle \nabla g_i(y),V_i(y) \rangle
+\E_{y,1} R^{\e,\delta}_i.
\ea
Therefore, we obtain (cf.\ the analogous expansion \eqref{e:opA1} for the operator $A^{\e,\delta}_1$)
\ba
\frac{\E_{y,1}f(Y^{\e,\delta}_{T_p^\e}) -f(y)}{\E_1 T^\e_p}
&=\frac{\varkappa^2}{\mu}
\Big(
\sum_{j<i} \mu_i\mu_j \langle \nabla \langle \nabla f(x),  V_i (x)\rangle, V_j(x) \rangle
+\frac{1}{2}  \sum_i  \lambda_i^2 \langle  \langle\nabla f(x),V_i(x)\rangle    , V_i (x)\rangle \Big)
+ r^{\e,\delta}_{x,1}.
\ea
Applying the identity
\ba
\langle \nabla \langle \nabla f,  V_i\rangle, V_j \rangle
= \langle \nabla^2 f  V_i , V_j \rangle + \langle \nabla f,  D V_i[V_j] \rangle, \quad 1\leq i,j\leq p,
\ea
yields the operator $B$ as defined in \eqref{e:opB}. The gliding dynamics produces an additional
$D V_i[V_i]$ drift component.
The remainder $R^{\e,\delta}_i$
is estimated in the same way as the corresponding Taylor remainders in Lemma \ref{l:A}. In particular,
its contribution after division by $\E_\iota T_p^\e$ converges to zero uniformly in $y$.
The argument for a general initial direction $\iota$ is identical to the cyclic-shift argument at the end of the proof of Lemma \ref{l:A}.
This proves the lemma.
\end{proof}

\begin{rem}
The difference between the limiting generators $A$ and $B$ originates
from the evolution of each vector field along its own flow. In the flight model
the vector field is frozen during a run, whereas in the gliding model this within-run variation
produces the additional $DV_i[V_i]$ contribution to the effective drift.
Thus, even for the same vector fields, run times, and cyclic switching rule,
the microscopic realization of the motion remains visible in the diffusion limit.
\end{rem}

Finally, we complete the proof of Theorem \ref{t:B} analogously to that of Theorem \ref{t:A}.
Indeed, every weak limit of the embedded processes $\widehat Z^{\e,\delta}$
solves the martingale problem for $B$. Since this martingale problem is well posed,
the embedded processes converge weakly to the diffusion $Y$ of Theorem \ref{t:B}.
Finally, the analogue of Lemma \ref{l:ucp} and the converging-together theorem yield $Y^{\e,\delta}\Rightarrow Y$,
which completes the proof.

The derivation of the SDEs \eqref{e:SDEY} and  \eqref{e:SDEYStrat}
proceeds analogously to the derivation in Theorem \ref{t:A}.

\section*{Acknowledgments}
O.A.\ acknowledges funding from the DFG project AR 1717/2-1 (548113512).

\section*{Disclaimer}
Spelling and grammar were checked, and the Python code used to simulate the
trajectories presented in Figure~\ref{f:planar} was developed, with assistance from ChatGPT.


\end{document}